\documentclass[preprint,3p,authoryear,11pt,fleqn]{elsarticle}
\usepackage[round]{natbib}

\usepackage{amsmath,amsthm}
\usepackage{amssymb}
\usepackage{dsfont}
\usepackage{amsfonts}
\usepackage{enumerate}
\usepackage{mathrsfs}
\usepackage{multirow}
\usepackage{float}
\usepackage{graphicx}
\usepackage{color}
\usepackage{endnotes}
\let\footnote=\endnote
\usepackage{abbreviationsSSC-ext}
\journal{Journal of Time Series Analysis}	

\begin{document}
\begin{frontmatter}

\title{Robust change point tests by bounded transformations}

\author[AD]{Alexander D{\"u}rre\corref{ad.cor}}
\address[AD]{Fakult\"at Statistik, Technische Universit\"at Dortmund, 44221 Dortmund, Germany}
\author[AD]{Roland Fried}

\cortext[ad.cor]{corresponding author: alexander.duerre@udo.edu, +49-231-755-4288}

\begin{abstract} 
Classical moment based change point tests like the cusum test are very powerful in case of
Gaussian time series with one change point but behave poorly under heavy
tailed distributions and corrupted data. A new class of robust change point tests based on
cusum statistics of robustly transformed observations is proposed. This framework is quite
flexible, depending on the used transformation one can detect for instance changes
in the mean, scale or dependence of a possibly multivariate time series. Simulations indicate that this approach is very powerful in detecting changes in the marginal variance of ARCH processes and outperforms existing proposals for detecting structural breaks in the dependence structure of heavy tailed multivariate time series.
\end{abstract}

\end{frontmatter}


\section{Introduction}
There is a fast growing literature on robust change-point detection. Maybe the first contribution is by \cite{page1955test} who proposed a sign based test, interestingly not with robustness in mind, but to get a distribution free procedure. Popular robust procedures are based on signs \citep{mcgilchrist1975note,vogel2015robust}, ranks \citep{bhattacharyya1968,horvath1994,gombay1998,antoch2008}, U-statistics \citep{pettitt1979,csorgo1988,gombay1995application,horvath1996,dehling2013change,dehling2015robust,vogel2015studentized}, quantiles \citep{csorgo1987}, M-estimators \citep{kumar1984,huskova1996,fiteni2002robust,huvskova2012} or distribution functions \citep{deshayes1986}.\\
We propose a general framework for robust change-point detection, using cusum statistics on robustly transformed observations. Denote therefore $\boldsymbol{X}_1,\ldots,\boldsymbol{X}_T$ a $p$-dimensional time series. We assume it to be stationary and short range dependent under the null hypothesis, see Section \ref{techassum} for the technical assumptions. If one is interested in a change of location and expects at most one change point it is common to look at the cusum statistic which for $p=1$ is defined as
\begin{align}\label{ocusum}
W_T(x)=\frac{1}{\sqrt{T}\hat{v}}\left(\sum_{i=1}^{\lfloor Tx \rfloor}X_i-\frac{\lfloor Tx \rfloor}{T}\sum_{i=1}^TX_i\right)
\end{align}
where $\hat{v}$ is a consistent estimator of the long run variance $v=\sum_{h=-\infty}^\infty \mbox{Cov}(X_i,X_{i+h})$ and rejects the null-hypothesis if functionals of $(W_T(x))_{x\in [0,1]}$ such as $\sup_{x \in [0,1]}|W_T(x)|$ are unusually large. If $p>1$ one often modifies (\ref{ocusum}) to the quadratic form
\begin{align}\label{quadcusum}
W_T^2(x)=\frac{1}{T}\left(\sum_{i=1}^{\lfloor Tx \rfloor}\boldsymbol{X}_i-\frac{\lfloor Tx \rfloor}{n}\sum_{i=1}^n\boldsymbol{X}_i\right)^T\hat{V}^{-1}\left(\sum_{i=1}^{\lfloor Tx \rfloor}\boldsymbol{X}_i-\frac{\lfloor Tx \rfloor}{n}\sum_{i=1}^T\boldsymbol{X}_i\right)
\end{align}
where $\hat{V}$ consistently estimates $V=\sum_{h=-\infty}^\infty \mbox{Cov}(\boldsymbol{X}_i,\boldsymbol{X}_{i+h}).$
To gain robustness we will transform the observations by a function $\Psi:\mathbb{R}^p\rightarrow \mathbb{R}^s$ which can be of very different kind to detect a shift for example in location, scale, skewness or dependence. Formal conditions on $\Psi$ can be found in Section \ref{techassum}. The main feature they have in common is the boundedness, we require that $|\Psi(\boldsymbol{x})|_\infty\leq b$ for some $b\in \mathbb{R},$ where $|~.~|_\infty$ denotes the maximum norm. This property is more an advantage than a restriction. By applying a bounded $\Psi$ we also bound the influence of outlying observations from another probability model than the bulk of the observations which could ruin the inference otherwise. An individual outlier can either fake a change-point although the other observations are stationary or also hide a true change-point. Furthermore a bounded $\Psi$ is also advantageous under heavy tailed distributions. Table \ref{examplepsi} contains a selection of useful $\Psi-$functions.\\
\begin{table}[H]
	\begin{center}
	\begin{tabular}{l|l}
		definition&measure\\\hline
		$\Psi_S(\boldsymbol{x})=s(\boldsymbol{x})$&location\\
		$\Psi_H(\boldsymbol{x})=\begin{cases}\boldsymbol{x},& |\boldsymbol{x}|\leq k\\
		s(\boldsymbol{x})k,&|\boldsymbol{x}|> k\end{cases}$&location\\
		$\Psi_{SCov}(\boldsymbol{x})=s(\boldsymbol{x})s(\boldsymbol{x})^T$&spread\\
		$\Psi_{HCov}(\boldsymbol{x})=\begin{cases}\boldsymbol{x}\boldsymbol{x}^T,& |\boldsymbol{x}|\leq k\\
		k^2s(\boldsymbol{x})s(\boldsymbol{x})^T,&|\boldsymbol{x}|> k\end{cases}$&spread\\
		\end{tabular}
	\caption{Selection of useful transformations $\Psi$. The function $s:\mathbb{R}^p\rightarrow R^p$ with $s(\boldsymbol{x})=\boldsymbol{x}/|\boldsymbol{x}|$ for $\boldsymbol{x}\neq 0$ and $s(\boldsymbol{x})=0$ for $x=0$ denotes the spatial sign.\label{examplepsi}}
		\end{center}
\end{table}
Note that these transformations need to be applied to properly standardized random variables. For example the univariate sign function applied to strictly positive random variables will destroy any information of the dataset. To prohibit such behavior we standardize the data marginally by location and scale estimators $\hat{\mu}^{(j)},~\hat{\sigma}^{(j)},~j=1\ldots,p$ and then apply the transformation $\Psi$:
\begin{align}\label{trafo}
\boldsymbol{Y}_{i,T}=\Psi([X_i^{(1)}-\hat{\mu}^{(1)}]/\hat{\sigma}^{(1)},\ldots,[X_i^{(k)}-\hat{\mu}^{(k)}]/\hat{\sigma}^{(k)})'=\Psi(D_{\hat{\boldsymbol{\sigma}}}^{-1}[\boldsymbol{X}_i-\hat{\boldsymbol{\mu}}]),
\end{align}
where $D_{\hat{\boldsymbol{\sigma}}}$ is the diagonal matrix containing $\hat{\sigma}_1,\ldots,\hat{\sigma}_p$ and $\hat{\boldsymbol{\mu}}=(\hat{\mu}_1,\ldots,\hat{\mu}_p)'.$
For now we only assume that these estimators converge in probability to some population values $\mu_j$ and $\sigma_j$ for $j=1,\ldots,k$ and postpone the discussion of the theoretical properties to Section \ref{techassum}. In practice we recommend the highly robust median and median absolute deviation. After the transformation we apply the quadratic cusum statistic (\ref{quadcusum}) to the transformed values:
\begin{align}\label{quadrobcusum}
W_T^2(x)=\frac{1}{T}\left(\sum_{i=1}^{\lfloor Tx \rfloor}\boldsymbol{Y}_{i,T}-\frac{\lfloor Tx \rfloor}{T}\sum_{i=1}^T\boldsymbol{Y}_{i,T}\right)^T\hat{U}^{-1}\left(\sum_{i=1}^{\lfloor Tx \rfloor}\boldsymbol{Y}_{i,T}-\frac{\lfloor Tx \rfloor}{T}\sum_{i=1}^T\boldsymbol{Y}_{i,T}\right).
\end{align}
Here $\hat{U}$ is an estimator for the respective long run variance $U=\mbox{Cov}\sum_{h=-\infty}^\infty \mbox{Cov}(\boldsymbol{Y}_i,\boldsymbol{Y}_{i+h})$,
where  
\begin{align*}
\boldsymbol{Y}_{i}=\Psi([X_i^{(1)}-\mu^{(1)}]/\sigma^{(1)},\ldots,[X_i^{(k)}-\mu^{(k)}]/\sigma^{(k)})'=\Psi(D_{\boldsymbol{\sigma}}^{-1}[\boldsymbol{X}_i-\boldsymbol{\mu}])
\end{align*}
with $D_{\boldsymbol{\sigma}}=\mbox{diag}(\sigma_1,\ldots,\sigma_p)$ and $\boldsymbol{\mu}=(\mu_1,\ldots,\mu_p)'$. One can either choose bootstrap methods \citep{hardle2003bootstrap,buhlmann2002bootstraps,kreiss2011bootstrap}, sub-sampling \citep{dehling2013estimation} or kernel estimators \citep{parzen1957consistent,andrews1991heteroskedasticity}. We propose the latter. For a bandwidth $b_T \in \mathbb{R}_+$ and a kernel function $k:\mathbb{R}_+\rightarrow \mathbb{R}_+$ it is defined as
\begin{align*}
\hat{U}=\frac{1}{T}\sum_{i,j=1}^{T}(\boldsymbol{Y}_{i,T}-\overline{\boldsymbol{Y}}_T)(\boldsymbol{Y}_{j,T}-\overline{\boldsymbol{Y}}_T)^Tk\left(\frac{|i-j|}{b_T}\right)
\end{align*}
where $\overline{\boldsymbol{Y}}_T$ is the arithmetic mean of $\boldsymbol{Y}_{1,T},\ldots,\boldsymbol{Y}_{T,T}$. We use the flat-top-kernel
\begin{align}\label{flattop}
k_F(x)=\begin{cases}
1&|x|\leq 0.5\\
2-2|x|&0.5<|x|\leq 1\\
0&|x|>1
\end{cases}
\end{align}
 which was first proposed in \cite{politis2001nonparametric}.
Our simulations reveal that the optimal bandwidth crucially depends on the dimension $s$ of $\boldsymbol{Y}_{1,T},\ldots,\boldsymbol{Y}_{T,T}.$ For $s=1$ we get good results with $b_T=0.9 T^{1/3}$, if $s>1$ the right choice of $b_T$ gets more complicated. Generally cusum tests get conservative under positive serial correlation and we notice that this effect gets stronger with increasing $s.$ We propose a very short bandwidth $b_T=\log_{1.8+s/40}(T/50)$ which compensates indirectly by underestimating the serial correlation. A comparable bandwidth is also chosen in \cite{aue2009break} where only a very light serial correlation is considered in the simulations.\\
We reject the null hypothesis of a stationary time series if $
M_T=\sup_{x\in [0,1]}W_T^2(x)$
is unusually large. We will show in Section \ref{techassum} that under stationarity $W_T^2(x)$ converges weakly to a well investigated stochastic process $(W^2(x))_{x\in [0,1]}$, often called Bessel bridge \citep{pitman1999law}, see Theorem 1 of Section \ref{techassum} for the necessary assumptions. Suitable asymptotical critical values of $M_T$ were first tabulated by \cite{kiefer1959k} for $p\leq 5$ and are implemented in the R-package robcp \citep{robcp} up to $p=5000$. Crude approximations for large $p$ are given in \cite{aue2009break}. A small selection of critical values can be found in Table \ref{tablequan}.
\begin{table}[H]
	\begin{center}
		\begin{tabular}{l|ccccccccc}
			$\alpha$/$p$&1&2&3&4&5&10&20&50&100\\\hline
			0.9&1.500&2.114&2.623&3.083&3.514&5.450&8.885&18.172&32.624\\
			0.95&1.844&2.508&3.053&3.543&4.000&6.041&9.626&19.219&34.022\\
			0.975&2.191&2.894&3.469&3.984&4.464&6.595&10.310&20.168&35.276\\
			0.99&2.649&3.396&4.004&4.548&5.053&7.288&11.154&21.321&36.783
		\end{tabular}
		\caption{Quantiles of $\sup_{0<x<1}W_T^2(x).$\label{tablequan}}
	\end{center}
\end{table}
Note that the supremum of $(W_T^2(x))_{x\in [0,1]}$ is most sensitive to changes in the middle of the time series. There are other functionals like $\int_0^1 W_T(x)^2~dx$ or weighted suprema $\sup_{0<x<1}W_T^2(x)/q(x)$ with $q(x)\rightarrow 0$ for $x\rightarrow 0$ and $x\rightarrow 1$ which are more powerful if changes occur at the beginning or at the end.\\
We conclude this section by mentioning some similar approaches in the literature. \cite{koul2003asymptotics} are among the first who considered M-estimators in the change point context, more precisely for estimating a change point in a regression model with iid errors. In \cite{fiteni2002robust} additionally to \cite{koul2003asymptotics} short range dependence is considered and the use of a standardization by a scale estimator $\hat{\sigma}$. \cite{han2006truncating} use truncated observations to estimate the time of a location change of a strongly mixing time series with heavy tails. M-estimators are closely connected to our theory. In the one-dimensional location context they are defined as solution of
\begin{align}
\argmin_{\mu \in \mathbb{R}} \sum_{t=1}^T \rho\left(\frac{X_t-\mu}{\hat{\sigma}} \right)\label{Mest1}
\end{align}
where $\hat{\sigma}$ is a scale estimator and $\rho$ a positive usually symmetric and and often convex function. If the latter is true one can reformulate (\ref{Mest1}) such that $\hat{\mu}$ is the unique solution of
\begin{align}\label{Mest2}
\sum_{t=1}^T \Psi\left(\frac{X_t-\mu}{\hat{\sigma}} \right)=0
\end{align} 
where $\Psi$ is the derivative of $\rho.$
So if we choose $\hat{\mu}$ as M-estimator defined by $\Psi$ in (\ref{trafo}), (\ref{quadrobcusum}) is the square (or quadratic form) of the cusum statistic of M-residuals. Looking at univariate M-residuals for change point testing has already been proposed in \cite{huvskova2005bootstrap} and \cite{huvskova2012m} but without the beneficial scale standardization. Multiple change point detection under iid noise using M residuals (also without standardization) is considered in \cite{fearnhead2017changepoint}\\
To detect changes in the dependence structure of a multivariate time series \cite{vogel2015robust} basically propose (\ref{quadrobcusum}) with $\Psi(\boldsymbol{x})=s(\boldsymbol{x})s(\boldsymbol{x})^T$ but do not consider a standardization by location and scale.\\
The rest of the article is structured as follows. Section \ref{techassum} consists of the technical assumptions and the theoretical results, while Section \ref{app} contains some tests based on specific $\Psi$-functions and their performance in small simulation studies. All proofs can be found in the appendix.

\section{Theoretical results}
\label{techassum}
In this section we give theoretical justification of the asymptotic critical values and compile all conditions which are necessary for the asymptotical results. We start by defining the type of short range dependence we impose on the time series.
We assume that it is strongly mixing.
\begin{assumption}
	Let $(X_t)_{t\in \mathbb{N}}$ be strongly mixing with mixing constants $(a_k)_{k \in \mathbb{N}}$ fulfilling \\
	$\sum_{k=1}^\infty k^2 a_k <\infty.$   
\end{assumption}
\begin{remark}
Linear and GARCH processes are strongly mixing if the innovations posses a Lebesgue density \citep{chanda1974strong,lindner2009stationarity}. For an overview on different concepts of mixing and its properties see \cite{bradley2005basic}.
\end{remark}
Next we take a closer look at the transformation $\Psi:\mathbb{R}^p\rightarrow \mathbb{R}^s.$ In the following $D_{\boldsymbol{x}}\in \mathbb{R}^{p\times p}$ denotes the diagonal matrix with the elements of $\boldsymbol{x}$ on its diagonal. Besides the boundedness of $\Psi$ we demand that the map is not too "irregular" and do not lead to a degenerate long run covariance matrix:
\begin{assumption}\label{Grundassum}
Let $\Psi:\mathbb{R}^p\rightarrow \mathbb{R}^s$ be a function fulfilling:
\begin{enumerate}[a)]
	\item there exists $C \in \mathbb{R}$ such that $|\Psi(\boldsymbol{x})|_{\infty}<C~\forall \boldsymbol{x} \in \mathbb{R},$
	\item $\mbox{det}(U)>0$ where $$U=\mbox{Var}(\Psi\{D_{\boldsymbol{\sigma}}^{-1}(\boldsymbol{X}_1-\boldsymbol{\mu})\})+2\sum_{h=1}^\infty\mbox{Cov}(\Psi\{D_{\boldsymbol{\sigma}}^{-1}(\boldsymbol{X}_1-\boldsymbol{\mu})\},\Psi\{D_{\boldsymbol{\sigma}}^{-1}(\boldsymbol{X}_{1+h}-\boldsymbol{\mu})\})$$
	\item Every component of $\Psi$ is two times continuous differentiable in $\mathbb{R}\backslash D$ and there exists $C_1,C_2>0$ with $|{\Psi^{(i)}}'(\boldsymbol{x})^T\boldsymbol{x}|\leq C_1$ and $\boldsymbol{x}^T{\Psi^{(i)}}''(\boldsymbol{x})\boldsymbol{x}\leq C_2,~\forall \boldsymbol{x}\in \mathbb{R}^p\backslash C$ and $i=1,\ldots,s$
	\item The exemption set $D$ can have the following properties:
	\begin{enumerate}[i)]
	\item It can contain balls $A_{\epsilon}(a_1),\ldots A_{\epsilon}(a_n)$ with radius $\epsilon>0$ around finitely many singularities $a_1,\ldots,a_k$ such that there exists $C_3>0$ and $\delta>0$ with $\sup_{T}\sup_{\epsilon\geq |\boldsymbol{x}-a_j|\geq T^{-2p}}|{\Psi^{(i)}}'(\boldsymbol{x})| \leq C_3 T^{p-\delta}$ and $\sup_{T}\sup_{\epsilon \geq |\boldsymbol{x}-a_j|\geq T^{-2p}}|{\Psi^{(i)}}''(\boldsymbol{x})| \leq C_3 T^{2p-\delta}$ for $j=1,\ldots,k$ and $i=1,\ldots,p.$
	\item If $\Psi$ is Lipschitz continuous, it can contain a bounded hypersurface $B$ where $\Psi$ is not differentiable.
	\item If $\Psi$ is Lipschitz continuous and fulfils the following condition: $\exists K>0$ such that for arbitrary but fixed $x_1,\ldots,x_{k-1},x_{k+1},\ldots,x_p$ 
	\begin{align*}\Psi(x_1,\ldots,x_{k-1},a,x_{k+1},\ldots,x_p)=\Psi(x_1,\ldots,x_{k-1},b,x_{k+1},\ldots,x_p),~\forall a,b>K,\end{align*}
than $D$ can contain unbounded sets $B_1,\ldots,B_l$ where $B_i$ are hyperplanes of the form $B_i=\{(x_1,\ldots,x_{m-1},a_i,x_{m+1},\ldots,x_p):x_1,\ldots,x_{k-1},x_{k+1},\ldots,x_p\in \mathbb{R}\}$ where $\Psi$ is not differentiable.
	\item It can contain unbounded sets $E_1,\ldots,E_r$ where $E_i$ are hyperplanes of the form $E_i=\{(x_1,\ldots,x_{m-1},a_i,x_{m+1},\ldots,x_p):x_1,\ldots,x_{k-1},x_{k+1},\ldots,x_p\in \mathbb{R}\}$ where $\Psi$ is discontinuous as long as $\Psi'(\boldsymbol{x})=0$ for $\boldsymbol{x}\in \mathbb{R}^p\backslash (E_1\cup\ldots\cup E_r)$.
	\end{enumerate}
\end{enumerate}
\end{assumption}
\begin{remark}
\begin{itemize}
\item In Assumption $a)$ we actually need only finite essential supremum and infimum of $\Psi$. One can even drop the boundedness condition completely and demand finite moments and a faster decrease of the mixing constants $(a_k)_{k\in \mathbb{N}}$. However boundedness is a necessary assumption for robustness.
\item Assumption $b)$ guarantees that we do not have a degenerate limit process which could be a result of a degenerate $(\boldsymbol{X}_t)_{t\in \mathbb{R}}$ or an improper choice of $\Psi$, $\hat{\boldsymbol{\mu}}$ or $\hat{\boldsymbol{\sigma}}.$
\item If one estimates $\boldsymbol{\mu}$ and $\boldsymbol{\sigma}$ one needs a type of continuity of $\Psi$. In the more specific situation of $M$-residuals without scale standardization \cite{huvskova2012m} assume Lipschitz continuity of the derivative of $\Psi$ and  Lipschitz continuity of some $2+\delta$ moment. We assume two times differentiability at almost all points but allow for exemption sets which enables the use of discontinuous $\Psi$.
\item The multivariate sign function $\Psi_s(\boldsymbol{x})$ has a singularity in $0$. Allowing for the exemption points $a_1,\ldots,a_k$ in d) i) enables the use this function and also the spatial sign cross product $\Psi_{SCov}(\boldsymbol{x}).$
\item The multivariate Huber function $\Psi_H(\boldsymbol{x})$ is not differentiable at $\{\boldsymbol{x}\in \mathbb{R}:|\boldsymbol{x}|=k\}.$ This exemption set is allowed because of d) ii) which also enables the use of the Huber cross product $\Psi_{HCov}(\boldsymbol{x})$
\item The marginal Huber functions defined by $\Psi(\boldsymbol{x})^{(i)}=x^{(i)}I_{|x^{(i)}|\leq k}+k\cdot s(x^{(i)})I_{|x^{(i)}|> k}$ for $i=1,\ldots,p$ are allowed due to d) iii) likewise marginal Huber cross products.
\item Marginal signs are allowed due to d) iv).
\end{itemize}
\end{remark}
Furthermore we need that both the location and the scale estimator be $\sqrt{T}$ consistent for their theoretical counterparts.
\begin{assumption}
Let $(\boldsymbol{\mu}_T)_{T\in \mathbb{N}}$ and $(\boldsymbol{\sigma}_T)_{T\in \mathbb{N}}$ be two stochastic sequences fulfilling
$$\boldsymbol{\mu}_T-\boldsymbol{\mu}=O(T^{-\frac{1}{2}})~~\mbox{and}~~
\boldsymbol{\sigma}_T-\boldsymbol{\sigma}=O(T^{-\frac{1}{2}}).$$
\end{assumption}
\begin{remark}
Assumption 3 is rather weak. We actually do not need to know the population values $\boldsymbol{\mu}$ and $\boldsymbol{\sigma}$ and in practice we rarely will. This would require knowledge of the marginal distribution of $\boldsymbol{X}_1.$ Consistency of quantile based estimators like median and MAD follows from the existence of a Bahadur representation which was shown under strong mixing and continuous density of the innovations in \cite{yoshihara1995bahadur}.  
\end{remark}
Finally we need assumptions for the kernel estimator. They are nearly identical with these in \cite{jong2000consistency}:
\begin{assumption}\label{assukern}
Let $k:\mathbb{R}\rightarrow [-1,1]$ be a function which is continuous in 0 and has only finitely many discontinuities. Furthermore it fulfills $k(0)=1,$ $k(x)=k(-x)$ for $x\in \mathbb{R}$ and $\int_{-\infty}^\infty |k(x)|~dx<\infty$ as well as $\int_{-\infty}^{\infty}|\int_{-\infty}^{\infty}k(x)e^{-2\pi x\xi}~dx|~d\xi<\infty.$ For the bandwidth $b_T$ it holds that: $b_T\rightarrow \infty$ and $b_T/T^{1-\epsilon}\rightarrow 0$ for $T\rightarrow \infty$ and some $\epsilon>0$
\end{assumption}
The next theorem contains the main result of this paper and describes the asymptotic distribution of $(W_T(x)^2)_{x\in [0,1]}$ under the null hypothesis.
\begin{theorem}\label{haupt}
	Let Assumptions 1-3 hold then
	\begin{align*}
	(W_T(x)^2)_{x\in [0,1]}\stackrel{w}{\rightarrow} (\sum_{i=1}^p BB_i(x)^2)_{x\in[0,1]}
	\end{align*}
	where $(BB_i(x))_{x\in [0,1]},~i=1,\ldots,p$ are mutually independent standard Brownian Bridges.
\end{theorem}
The result concerns weak convergence in the Skorochod space $D[0,1]$ which consists of all functions which are right continuous with left-hand limits. The continuous mapping theorem therefore yields the validity of the proposed asymptotic critical values.\\
The linear structure of the test statistic makes it feasible to derive also asymptotics under the alternative. Nevertheless there are two challenges. The first one is the estimation of the location and scale standardization and the second is the estimation of the long run covariance. Especially if $s>1$ this gets tricky since we need to ensure positive definiteness of $V$ which depends on the direction of the alternative. If one looks at local alternatives this posses less problems, nevertheless we postpone the theory under alternatives to future work and investigate the properties of our approach by simulations.

\section{Application}\label{app}
Theorem \ref{haupt} enables various kinds of change point tests. As mentioned in the introduction one can detect amongst others changes in location scale and dependence. If one does not assume a specific model one will usually use truncated moments like in Table \ref{examplepsi}. Otherwise truncated scores of the log-likelihood can e used. For example, if $X_{1},\ldots,X_{T}$ is assumed to be a sequence of independent exponentially distributed random variables under the null hypothesis, one can use
\begin{align*}
\Psi(X_i)=\begin{cases}
\frac{\ln(2)X_i}{\hat \sigma},& \frac{\ln(2)X_i}{\hat \sigma}\leq k\\
k,& \frac{\ln(2)X_i}{\hat \sigma}>k
\end{cases}
\end{align*}
where $\hat{\sigma}=\mbox{Median}(X_1,\ldots,X_T).$ Note that we do not center the observations here and drop the term $-\frac{\log(2)}{\hat{\sigma}}$ since it does not change the cusum statistic. Even if the model is incorrectly specified, the observations are serially correlated or follow a different distribution the test is still valid with respect to the asymptotic size under the null hypothesis of no change but will lose power under the alternative.\\
In the following we present some useful non-parametric tests. If one tries to detect a change in location and $p=1$ we recommend the Huber-$\Psi$-function
\begin{align*}
\Psi_H(x)=\begin{cases}
x,&|x|\leq k\\
ks(x),&|x|>k
\end{cases},
\end{align*}
originally proposed for location estimation in \cite{huber1964robust}. Different authors propose different choices of $k$ in the estimation context. In \cite{huber2011robust} $k\in [1,2]$ is recommended and $k=1.5$ suggested, but there are also other proposals favour $k=1.2$ \citep{cantoni2001robust}, \cite{street1988note} $k=1.25$ and \cite{wang2007robust}  propose a data dependent $k.$ In general a larger value of $k$ is more efficient under normality but less efficient under heavy tails and outliers. We investigate the impact of $k$ in the related problem in detecting a scale or scatter shift.\\
There are two straightforward generalizations for $p>1.$ One can use $p$ univariate Huber-M-$\Psi$-functions 
\begin{align*}\tilde{\Psi}_H(\boldsymbol{x})=(\Psi_H(x)^{(1)},\ldots,\Psi_H^{(p)}(x))'
\end{align*} ore a multivariate one
\begin{align*}
\Psi_H(\boldsymbol{x})=\begin{cases}\boldsymbol{x},& |\boldsymbol{x}|\leq k\\
		s(\boldsymbol{x})k,&|\boldsymbol{x}|> k\end{cases}.
\end{align*}
One expects that the later version is more powerful but less robust in case of elliptical marginal distributions as this applies for the estimation problem \citep{maronna1976robust}. If $p$ is large compared to $T,$ tests based on projections might be preferable though the power of the test crucially depends on the  direction $\mathbf{a}\in \mathbb{R}^p$ chosen for the projection. The choice
\begin{align*}
\Psi_P(\boldsymbol{x})=(\Psi_{H}(x)^{(1)},\ldots,\Psi_{H}^{(p)}(x))\boldsymbol{a}
\end{align*}  
gives a robustified projection based test.\\
The sample variance and covariance are due to their quadratic nature even more influenced by outliers than the arithmetic mean. For symmetric distributions
\begin{align*}
\Psi_{HVar}(x)=\begin{cases}
x^2,&|x|\leq k\\
k^2,&|x|>k
\end{cases}
\end{align*}
is an intuitive choice to find a change in the scale of time series with rather symmetric marginal distributions, but there are many other possibilities. One can of course also look at Huberized absolute values or any power of it. If the marginal distribution is very heavy skewed the test is still valid but not optimal. If the distribution is for example positively skewed the $\Psi$-function will downweight many valuable observations on the right tail and on the other hand overlook outliers on the left tail. If the time series is additionally positive $\Psi_{LogHVar}(x)=\Psi_{HVar}{(log(x))}$ might be a suitable choice, since it can symmetrize the marginal distribution.\\
Like in the location case there are at least two intuitive generalizations for $p>1$. The first one Huberizes observations in every component individually
\begin{align*}
\tilde{\Psi}_{HCov}^{i,j}(\boldsymbol{x})=\Psi_{HVar}(x^{(i)})\Psi_{HVar}(x^{(j)}),
\end{align*}
while the second one Huberizes all components at once
\begin{align*}
\Psi_{HCov}(\boldsymbol{x})=\begin{cases}\boldsymbol{x}\boldsymbol{x}^T& |\boldsymbol{x}|\leq k\\
		k^2s(\boldsymbol{x})s(\boldsymbol{x})^T&|\boldsymbol{x}|> k\end{cases}.
\end{align*}
Depending on the outlier model one proposal is probably more favourable than the other. The first one is related to the scatter estimator proposed in \cite{van2016stahel}, which is constructed to cope with cellwise outliers. In this setting the elements of an observation vector are corrupted individually. The second approach is similar to scatter estimators recently proposed in \cite{raymaekers2018generalized} which show good results in case of rowwise outliers, where whole observations are corrupted. If an observation $\boldsymbol{X}_{t^\star}$ is corrupted only in a few components one should apply $\tilde{\Psi}_{HCov}$, while if one expects only heavy tailed observations or rowwise no outliers then $\Psi_{HCov}(\boldsymbol{x})$ should be used.\\
For transformations measuring dependence we need additional operators, since assumption 2b) is violated by default otherwise. If one simply vectorizes a symmetric matrix one gets duplicated elements and therefore perfectly linear dependence which lead to a singular $V.$ We define $DU:\mathbb{R}^{p\times p}\rightarrow \mathbb{R}^{p(p+1)/2}$ with $DU(M)=(M_{1,1},M_{2,1},\ldots,M_{p,1},M_{2,2},\ldots,M_{2,p},\ldots,M_{p,p})^T$ which extracts the diagonal and lower diagonal elements of a matrix. Then also the Huberized covariance transformation $DU(\Psi(x))$ with
$\Psi_{HCov}(\boldsymbol{x})$ fulfills Assumption 2. For the spatial sign covariance matrix we need $dU:\mathbb{R}^{p\times p}\rightarrow \mathbb{R}^{p(p+1)-1}$ with $dU(M)=(M_{1,1},M_{2,1},\ldots,M_{p,1},M_{2,2},\ldots,M_{2,p},\ldots,M_{p,p-1})^T$ which eaves the last diagonal element out as opposed to $DU$
\section{Simulations}\label{simu}
We want to evaluate advantages and disadvantages of the proposed approach in some simulations. In the following we concentrate on the case of a change in the variance ($p=1$) respectively covariance ($p>2$) of a possibly multivariate time series. If one is interested in a change in location we refer to \cite{dehling2015robust} which includes extensive simulations. Our approach is called Huberization test there and is quite competitive, though it is beaten by the two sample Hodges-Lehmann test proposed there. However, there are two advantages which might counterbalance a slight disadvantage with respect to power. First, our procedure has a computational complexity of $T\log (T)$ whereas for the current implementation of the Hodges-Lehmann based test it is of $T^3$ prohibiting its application to very large samples and second, our approach based on bounded transformations has a natural extension to the case $p>1.$ Under elliptical distributions it is known that estimators and tests based on the multivariate sign function get more efficient with increasing dimension \citep{paindaveine2016high} and this is also observed in the change point context \citep{vogel2015robust}.
\subsection{Change in marginal scale of a one dimensional time series}
We first want to take a closer look at the case $p=1$ and a change in the scale of a one dimensional time series. There are some tests in the literaturefor this situation which are robust to some degree or appropriate under heavy tails, see \cite{gerstenberger2016tests}. All three of them are cusum type tests based on different estimators of scale, using
\begin{align*}
\max_{k=1,\ldots,n}\frac{k}{\sqrt{n\hat{v}}}|s_{1:k}-s_{1:n}|
\end{align*}
where $s_{1:k}$ is a scale estimator based on $X_1,\ldots,X_k$ and $\hat{v}$ is an estimator of the asymptotic variance of $s_{1:n}.$ In \cite{gerstenberger2016tests} three estimators are investigated. The mean absolute deviation (abbreviated as MD) is defined as
\begin{align*}
s_{1:n}=\frac{1}{n}\sum_{i=1}^n|x_i-q_{0.5}(X)|
\end{align*}
where $q_\alpha(X)$ denotes the $\alpha-$Quantile of $X_1,\ldots,X_n.$ The MD has an asymptotic efficiency of 0.88 compared to the standard deviation under independent and identically normal distributed data \citep{fisher1921}, but its breakdown point is 0. Ginis mean difference (GMD) is a $U-$statistic of the following form
\begin{align*}
s_{1:n}=\frac{2}{n(n-1)}\sum_{1\leq i<j\leq n}^n|X_i-X_j|.
\end{align*} 
It has an efficiency of 0.98 \citep{gerstenberger2015efficiency} and also a breakdown point of 0. The quantile of pairwise differences ($Q_n^{(\alpha)}$) is a U-Quantile defined as
\begin{align*}
Q_n^{(\alpha)}=q_{\alpha}(|X_i-X_j|:1\leq i<j\leq n).
\end{align*}
To get a robust estimator \cite{RousseeuwCroux1993} propose $\alpha={\lfloor n/2 \rfloor +1  \choose 2}\approx 0.25$ which yields an optimal breakdown point of 0.5 and an efficiency of 0.82. However, \cite{gerstenberger2016tests} explore that $\alpha=0.8$ is more appropriate in the change point setting. This yields an efficiency of 0.96 and a breakdown point of 0.08.\\
Our test is based on the psi-function $\Psi_{HCov}.$ It remains to choose a suitable value of $k.$ A larger value transforms less values and leads to more powerful tests under Gaussianity whereas a smaller $k$ is more efficient under heavy tails. In the following we will choose $k$ as a quantile of the $\chi^2_1$ distribution (with one degree of freedom). So $k=q_{\chi^2_1}(0.5)\approx 0.45$ means that under independent normally distributed data approximately 50\% of the observations is transformed. We choose the 0.5, 0.8 and 0.95 quantile and denote the resulting estimators as $M_{0.5},~M_{0.8}$ respectively $M_{0.95}.$\\
In our simulations we concentrate on ARCH models of order one:
\begin{align}\label{ARCH}
X_t=\sigma_t\epsilon_t,~~~t=1,\ldots,T
\end{align}
where $(\epsilon_t)_{t\in \mathbb{Z}}$ is a series of iid Gaussian random variables and $\sigma_t^2=\pi_0+\pi_1 X_{t-1}^2.$ This comparatively simple model allows us to investigate the effect of heavy tails as well as serial correlation. A large value of $\pi_1$ yields a strong dependence and heavy tails, while $\pi_1=0$ produces an iid Gaussian time series.\\
There are some specific tests to detect changes in ARCH models. We included some of these in our simulations. The first proposal is by \cite{kokoszka2002change} who propose a test on estimated GARCH residuals $\hat{\epsilon}_1,\ldots,\hat{\epsilon}_T.$ We abbreviate it as GARCH-res-ecdf. Denote $\hat{F}_{1:k}$ the empirical distribution function of $\hat{\epsilon}_1^2,\ldots,\hat{\epsilon}_k^2$ and by $\hat{F}_{k+1:n}$ the one based on $\hat{\epsilon}_1^2,\ldots,\hat{\epsilon}_k^2.$ Then the test statistic is computed as
\begin{align*}
B_T=\frac{1}{T^2}\sum_{i=1}^T\sum_{j=1}^T\left(\sqrt{T}\frac{i}{T}\left[1-\frac{i}{T}\right]|\hat{F}_{1:i}[\hat{\epsilon}_j]-\hat{F}_{i+1:T}[\hat{\epsilon}_j]|\right)^2.
\end{align*}
Asymptotical critical values of $B_T$ can be found in \cite{kokoszka2002change} which date back to \cite{blum1961distribution}.\\
Also residual based is the approach of \cite{kulperger2005high} which uses 
\begin{align*}
D_T=\max_{t=1,\ldots,T}\frac{|\sum_{i=1}^t(\hat{\epsilon}_i- \overline{\hat{\epsilon}})^2-i\hat{\sigma}^2|}{\sqrt{n}\hat{v}},
\end{align*}
where 
\begin{align*}\overline{\hat{\epsilon}}=\frac{1}{n}\sum_{t=1}^T \hat{\epsilon}_t,~~
\hat{\sigma}^2=\frac{1}{n}\sum_{t=1}^n(\hat{\epsilon}_t-\overline{\hat{\epsilon}})^2~~\mbox{and}~~
\hat{v}^2=\frac{1}{n}\sum_{t=1}^T[(\hat{\epsilon}_t-\overline{\hat{\epsilon}})^2-\hat{\sigma}^2]^2.
\end{align*}
The cusum-type statistic $D_T$ is asymptotically distributed like the maximum of the absolute value of a Brownian bridge. We abbreviate this test by GARCH-res.\\
There are two more tests based on the likelihood function. Both require the parameter $\pi_1$ to be strictly larger than 0. The first one is by \cite{berkes2004testing} who propose using functionals of
\begin{align}\label{Garchmul}
E(t)=\frac{1}{T}\left[\sum_{i=1}^{\lfloor nt \rfloor} \hat{l}_i'(\hat{\boldsymbol{\pi}}) \right]\hat{D}^{-1}\left[\sum_{i=1}^{\lfloor nt \rfloor} \hat{l}_i'(\hat{\boldsymbol{\pi}}) \right]^T.
\end{align}
Thereby $l_i'$ denotes the estimated score. If one assumes normally distributed errors $\epsilon_t,~t=1,\ldots,T$ for the stated model (\ref{ARCH}) it equals
\begin{align*}
l_i'(\hat{\boldsymbol{\pi}})=\left(\frac{X_i^2}{2(\hat{\pi}_0+\hat{\pi}_1X_{i-1})^2}-\frac{1}{2(\hat{\pi}_0+\hat{\pi}_1X_{i-1})},\frac{X_i^2X_{i-1}^2}{2(\hat{\pi}_0+\hat{\pi}_1X_{i-1})^2}-\frac{X_{i-1}^2}{2(\hat{\pi}_0+\hat{\pi}_1X_{i-1})}\right).
\end{align*}
and $D$ can be estimated by $\frac{1}{T}\sum_{t=1}^T \hat{l}_i'(\hat{\boldsymbol{\pi}})^T\hat{l}_i'(\hat{\boldsymbol{\pi}}).$ It is shown that $(E(t))_{t\in [0,1]}$ converges to a squared Bessel bridge with parameter 2. To simplify the computational complexity we use the maximum of $E_t$. We call the Test GARCH-LM3 in the following.\\
A Lagrange multiplier test which enables detecting changes of single parameters is proposed by \cite{galeano2009shifts} and abbreviated as GARCH-LM1. The test statistic reads
\begin{align*}
LM_T=\max_{t_{min} \leq t \leq t_{max}} \frac{T}{t(1-t)}\frac{[\overline{l}_tB_T(\hat{\boldsymbol{\pi}})A_T(\hat{\boldsymbol{\pi}})_{[2]}]^2}{A_T(\hat{\boldsymbol{\pi}})_{[2]}B_T(\hat{\boldsymbol{\pi}})^{-1}A_T(\hat{\boldsymbol{\pi}})_{[2]}^T}.
\end{align*} 
where $\overline{l}_t=\frac{1}{T}\sum_{i=1}^t l_i(\hat{\boldsymbol{\theta}})$. In our model (\ref{ARCH}) the scores $l_i(\hat{\boldsymbol{\theta}})$ have the following form
\begin{align*}
l_i(\hat{\boldsymbol{\theta}})=\frac{X_i-\hat{\mu}}{\hat{\pi}_0+\hat{\pi}_1(X_{i-1}-\hat{\mu})^2}(1,0,0)^T+\frac{1}{2}\left(\frac{(X_i-\hat{\mu})^2-\hat{\pi}_0-\hat{\pi}_1(X_{i-1}-\hat{\mu})^2}{[\hat{\pi}_0+\hat{\pi}_1(X_{i-1}-\hat{\mu})^2]^2}\right)h'_i(\hat{\boldsymbol{\theta}})^T
\end{align*}
with
\begin{align*}
h'_i(\hat{\boldsymbol{\theta}})=[-2(X_{i-1}-\hat{\mu}),1,(X_{i-1}-\hat{\mu})^2].
\end{align*}
Furthermore we have
\begin{align*}
A_T(\hat{\boldsymbol{\theta}})_{[2]}=-\frac{1}{T}\sum_{t=1}^T\left[\frac{1}{2[\hat{\pi}_0+\hat{\pi}_1(X_{i-1}-\hat{\mu})^2]^2}h'_i(\hat{\boldsymbol{\theta}})^Th'_i(\hat{\boldsymbol{\theta}})\right]_{[2]}~~\mbox{and}~~B_T=\frac{1}{T}\sum_{t=1}^Tl_i(\hat{\boldsymbol{\theta}})l_i(\hat{\boldsymbol{\theta}})^T
\end{align*}
with $(\cdot)_{[i]}$ denoting the $i$-th column of a matrix. By choosing the second column of $A_T$ we test against a change in $\pi_0$ which reflects a change in the marginal variation. In \cite{de1981crossing} the following asymptotic formula to compute p-values of $LM_T$ is derived
\begin{align*}
P(LM_T>x)\approx\frac{1}{\Gamma\left(\frac{1}{2}\right)}\sqrt{2x}e^{-\frac{x}{2}}\left[\log\left(\frac{t_{\max}}{t_{\min}}\right)\left(1-\frac{1}{x}\right)+\frac{2}{x}\right].
\end{align*}
The values $t_{\min}=0.05$ and $t_{\max}=0.95$ are chosen as boundary values.\\
For all four tests we use the R-package fGarch \citep{wuertz2009fgarch} to estimate the parameters of the ARCH(1) model.\\
We have not addressed the specific choice of the kernel and the bandwidth of the long run variance estimation yet. Assumption \ref{assukern} is very general. From the theory of spectrum estimation we know that a large choices of $b_n$ yield a large variance whereas small ones can produce a large bias. There are optimal rates depending on the time series model and plugin estimators if the model is unknown. Note that these data dependent estimators are constructed under the null hypothesis. Under the alternative they overestimate the dependence and yield very large vales of $b_n$ which seriously affect the power \citep{vogelsang1999sources}. We tried a large number of different fixed bandwidths and found that the ones in Table \ref{bandw} to be especially useful for ARCH(1) models.
\begin{table}
\begin{center}
	\begin{tabular}{l|c|c}
estimator&$b_T$&kernel\\\hline
$M_q$&$0.9T^{1/3}$&\multirow{4}{*}{$k(x)=\begin{cases}1&x\leq 0.5\\
2-2x& 0.5<x<1\end{cases}$}\\
MD&$T^{1/4}$&\\
GMD&$T^{1/4}$&\\
$Qn^{(0.8)}$&$T^{1/3}$&
	\end{tabular}
	\caption{Choosen bandwidths $b_T$ and kernels $k$ for the estimation of the long run variance.\label{bandw}}
	\end{center}
\end{table}
Note that a small bandwidth is always preferable under the alternative, since a level shift results in large estimated autocorrelations and therefore a large estimated long run variance, which enters the denominator of the test statistic. In this regard MD and GMD have a slight advantage over $M_k$ and $Qn^{(0.8)}$. The latter requires an additional tuning parameter for a necessary density estimation. We used the same as \cite{gerstenberger2016tests} but noticed that this nuisance parameter is not very influential on the power and the size of the test.\\
Finally we use the finite sample correction proposed in \cite{durre2018}, leading to 
\begin{align}\label{fini}
\sup_{x\in [0,1]}|W_T(x)|+\frac{\zeta(0.5)}{\sqrt{2\pi\cdot T}}.
\end{align}
The summand $\zeta(0.5)/\sqrt{2\pi\cdot T}\approx 0.58/\sqrt{T}$ originally results from the asymptotic expectation of 
\begin{align*}\sqrt{T}\left(\sup_{x\in [0,1]}B_x-\max_{x=0,1/T,\ldots,1} B_x\right),\end{align*}
see \cite{asmussen1995discretization}. But it turns out that this finite sample correction is also valid for the maximum of the absolute values of a Brownian bridge \citep{durre2018}. Simulations indicate that this correction is also useful for robustly transformed tests, see Figure \ref{kor}. 
\begin{figure}
\begin{center}
	\includegraphics[width=0.95\textwidth]{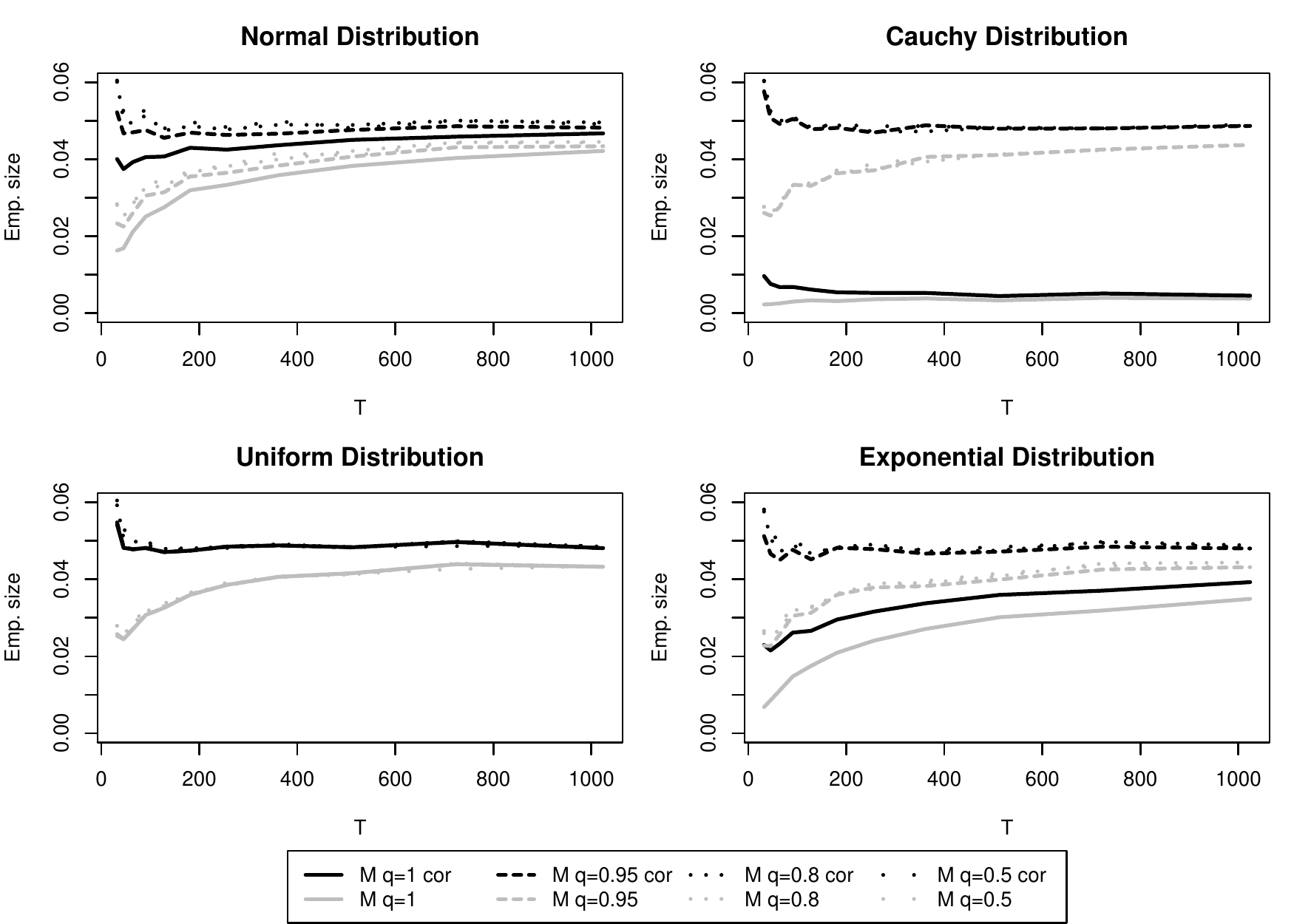}
	\caption{Simulated empirical sizes of $M_q$ tests under $X_T=\epsilon_t,~t=1,\ldots,T.$. where $\epsilon_t$ are independent and identically distributed following a standard normal distribution (top-left), a Cauchy distribution (top-right), a uniform distribution on $[0,1]$ (bottom-left) and an exponential distribution with parameter $\lambda=1$ (bottom right).\label{kor}}
	\end{center}
\end{figure}
It turns out that the correction is also useful for $MD$, $GMD$ and $Qn^{(0.8)}$, which is not surprising since the linearization of all these tests is the ordinary cusum test. Therefore we added the correction also to these tests.\\ 
First we want to asses the size under the null hypothesis.
\begin{figure}
\begin{center}
	\includegraphics[width=0.8\textwidth]{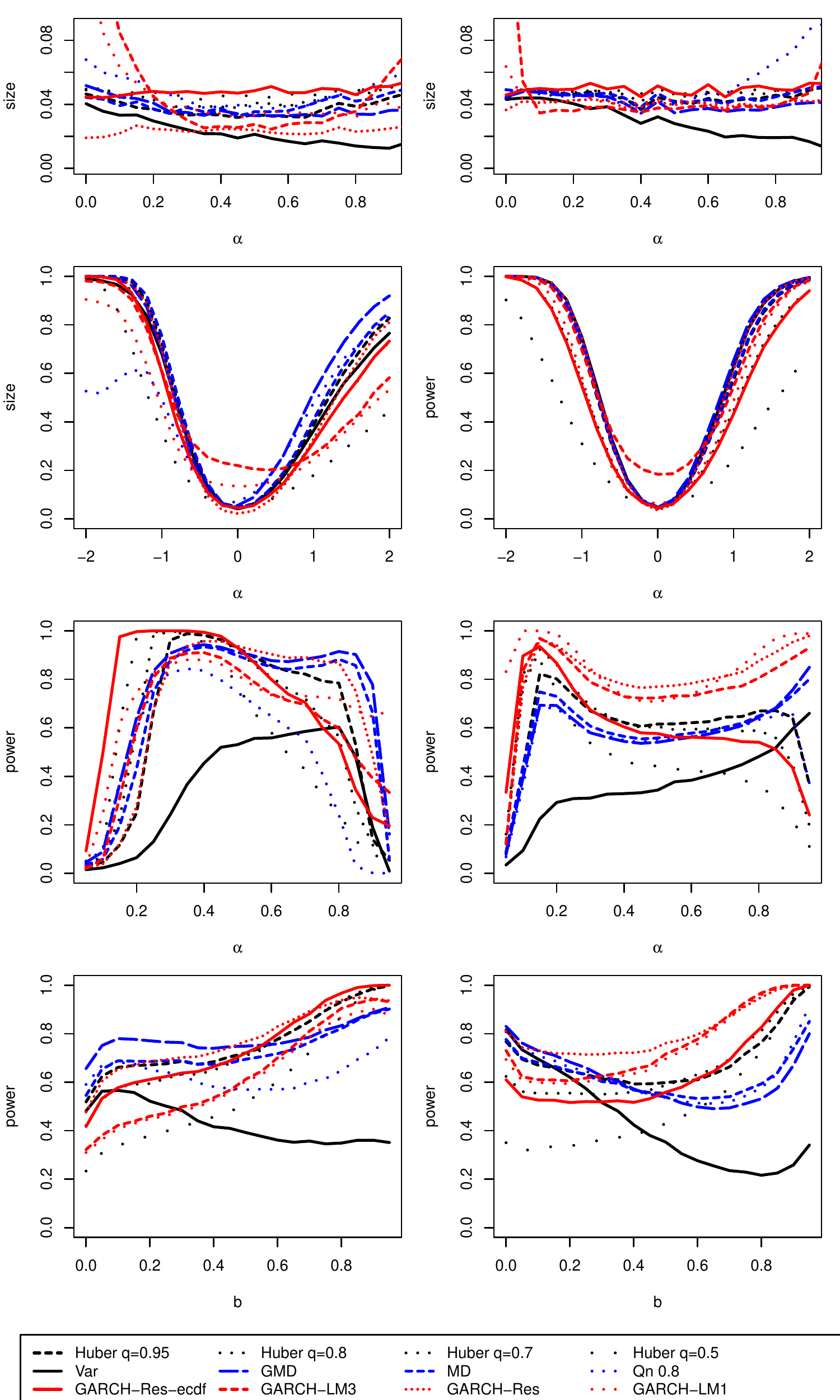}
	\caption{Simulation results under ARCH(1) model $X_t=\epsilon \sigma_t$ with $\sigma_t=\pi_0+\pi_1X_{t_1},~t=1,\ldots,T$ with $T=100$ (left) and $T=800$ (right). The first row shows the empirical size under $H_0$ with respect to different $\pi_1.$ The remaining rows show the power with respect to different variables, namely the size of the change (2nd row), the time of the change-point (3rd row) and $\pi_1$ (4th row).\label{figgarch}}
	\end{center}
	\end{figure}
The first row of Figure \ref{figgarch} reveals that almost all tests hold their size irrespective of $\pi_1$ (the larger $\pi_1$ the larger the serial correlation) and the time series length $T$. For small sample sizes we observe that the non parametric tests (blue and black curves) get a little conservative for moderately large $\pi_1$. This behaviour vanishes as $\pi_1$ increases further. This is typical for cusum type tests, since these tests get more conservative with increasing serial dependence \citep{durre2018}. On the other hand under increasing serial dependence the estimation of the long run variance gets more negatively biased, which first cancels out the conservativeness and then dominates it.\\
We notice that $Q_n^{(0.8)}$ exceeds its size if both $\pi_1$ and $T$ are large. This points to a too small choice of the bandwidth. Since $Q_n^{(0.8)}$ has already the largest bandwidth and \cite{gerstenberger2016tests} also propose usage $T^{-1/3}$ we decide against enlarging $b_T$ which would considerably harm power under the alternative. It is also noteworthy that $M_\infty$ is very conservative for large $\pi_1.$ This is not surprising, since the assumption of finite second moments, which are in fact fourth moments testing the variance testing, is violated. Furthermore we notice that the likelihood based tests have problems if $\pi_1$ is close to 0, which is not surprising since $\pi_1>0$ is an assumption for the asymptotics. GARCH-res is rather conservative under any value of $\pi_1,$ especially if $T$ is small. The finite sample correction (\ref{fini}) might also be advantageous here.\\
In the following we evaluate the power under the alternative. We use the model:
\begin{align*}
Y_t=\begin{cases}
X_t,&t\leq bT\\
X_t \cdot \delta,&T>bT
\end{cases}
\end{align*}
where $(X_t)_{t\in 1,\ldots,T}$ follows the ARCH(1) model defined in (\ref{ARCH}). Generally we want to compensate for obvious effects and therefore relate the jump height $\delta$ to the sample size $T$, the fraction of the data before the change $b$ and the ARCH-Parameter $\pi_1$. 
One expects asymptotically stable power of the tests for jump heights which are proportional to $1/\sqrt{T}$. Furthermore the power should decrease if the value $b$ departs fron 0.5. It turns out that a jump height proportional to $1/b/(1-b)$ stabilize power. Finally the power should decrease with increasing $\pi_1$ due to heavier tails and more serial correlation. The accurate stabilizing function is not known and differs for the different estimators. We  account for this with the factor $1/(1-\pi_1)$ resulting in
\begin{align}\label{jumpheight}
\delta=1+\frac{s}{\sqrt{T}\cdot b\cdot (1-b)\cdot (1-\pi_1)}.
\end{align}
Using this more complicated jump heights allows us to see more clearly the differences between the estimators and characteristics of their finite sample behaviour.\\
First we want to investigate the effect of increasing $s$.  We set $\pi_1=0,~b=0.5$ and $T=100$ and visualize the power under varying jump height $s$ on the left hand side of the second row of Figure \ref{figgarch}. We see that the power under a negative jump of size $s$ is larger than the power under a positive $s.$ This might be due to the fact that the variance respectively scale lives naturally on a log scale. However if $T$ is very large $\delta$ differs only marginally from 1 where the derivative of the logarithm is nearly 1, so we do not see this effect for $T=800$ on the right hand side of the second row of Figure \ref{figgarch}. More surprisingly is the fact that the order of the tests depends on the sign of the jump. For a negative jump MD has the largest power, followed by $M_{q_{0.95}}$ and GMD. The ordering changes to GMD, $M_{q_{0.95}}$ and MD for a positive jump. The variance based test $M_{\infty}$ has problems for small $T$ due to its conservativeness under the null hypotheses but becomes the most powerful test under $T=800.$ Noteworthy is also the non-monotonic behaviour of $Q_n^{(0.8)}$ for negative jumps and small $T$. The specific GARCH-tests are generally not so powerful as the non-parametric ones, which is due to the choice of $\pi_1=0$. We see furthermore that the parametric tests gain more power with increasing $T$.\\
In the third row of Figure \ref{figgarch} we see the effect of the fraction of change $b$. We fix $\pi_1=0.5$ and $s=1$ and vary $b$ between 0.05 and 0.95 in steps of 0.05. Note that one expects the maximal power at $b=0.5$. Since we already account for this in the jump height (\ref{jumpheight}), one rather expects a straight line instead. Especially for $T=100$ we see large deviations from that. GMD, MD, $M_\infty$ and GARCH-LM1 are more powerful if the change appears at the end of the time series than in the beginning. The other tests show the opposite behaviour. The most extreme here is the GARCH-res-ecdf followd by $M_{q_{0.5}}$ and $Q_n^{(0.8)}$. The larger $k$ the less pronounced is the difference in power between early and late changes. For $T=800$ we see a larger plateau where the power is indeed constant, but the asymmetry of the tests is still visible. Specific GARCH tests generally outperform their competitors now because of the choice $\pi_1=0.5$\\
In the last row of Figure \ref{figgarch} we investigate the tests under increasing serial correlation and heavy tails by varying $\pi_1$ between 0 and 0.95 in steps of 0.05. We set $s=1$ and $b=0.5$. Results for $T=100$ are shown on the left hand side and for $T=800$ on the right. Since the jump is positive GMD is most powerful for small $\pi_1$. As before $M_\infty$ is handicapped by its conservativeness under the null and therefore for Gaussian time series only as powerful as GMD for $T=800$. We see furthermore that for $p_1>0.55$ for $T=100$ respectively $pi_1>0.35$ for $T=800$ $M_{q_{0.95}}$ becomes the most powerful non parametric test. For $T$ and $\pi_1$ large it gets beaten by $M_{q_{0.95}}$ though. The parametric tests outperform the non-parametric ones if $\pi_1>0.4$. The difference gets more pronounced if $T$ is large, as we have seen before. Generally the residual based test GARCH-res seems to be the best choice overall.\\
In summary we have seen that our approach is advantageous under heavy tails compared to other non-parametric methods. The tuning parameter $k$ should be chosen according to the degree of heavy tailedness. Without a-priori information we recommend using $k=q_{\chi^2}(0.95)$ since it delivers good results under normality and various degrees of heavy tailedness. If one can assume Gaussianity we prefer MD over GMD because of its computational simplicity. If one expects a GARCH process with severe serial dependence one should use GARCH-res.
\subsection{Change in cross sectional dependence of a multivariate time series}
Now we want to see how results generalize for $p>1.$ To the best of our knowledge there is no other robust test for a change in the cross covariance $\Sigma=\mbox{Cov}(\boldsymbol{X}_t)$. But there is a non-robust one by \cite{aue2009break} based on the empirical covariance (abbreviated as Cov)
\begin{align*}
\frac{1}{n}\sum_{i=1}^n (\boldsymbol{X}_i-\hat{\boldsymbol{\mu}})(\boldsymbol{X}_i-\hat{\boldsymbol{\mu}})^T.
\end{align*}
where $\hat{\boldsymbol{\mu}}$ is the arithmetic mean. Except for using the mean instead of the median this test is a special case of our Huberized covariance test with $k=\infty.$ There is some kind of robustified version which uses roots of absolute values of the original observations:
$\boldsymbol{Y}_i=(Y_i^{(1)},\ldots,Y_i^{(p)})=(|X_i^{(1)}|^\delta,\ldots,|X_i^{(p)}|^\delta)$ with $\delta\in (0,1]$ but in our simulations with random vectors which are not positive this approach does not lead to decent results. This is not surprising since taking absolute values of the components completely destroys the dependence structure.\\
There are other approaches which concentrate on a change in the dependence structure. In \cite{wied2014nonparametric} a test based on the empirical correlation (abbreviated as Cor) is proposed
\begin{align*}
\hat{\rho}_{k,l}=\frac{\sum_{i=1}^T(X_i^{(k)}-\overline{X^{(k)}})(X_i^{(l)}-\overline{X^{(l)}})}{\sqrt{\sum_{i=1}^n(X_i^{(k)}-\overline{X^{(k)}})^2\sum_{i=1}^n(X_i^{(l)}-\overline{X^{(l)}})^2}}
\end{align*}
The serial correlation is accounted by a block bootstrap. The sample correlation is known to be efficient under normality but not robust.\\
A more robust approach is presented in \cite{bucher2014}, who propose a test based on the empirical copula
\begin{align*}
\hat{C}(\boldsymbol{u})=\frac{1}{n}\sum_{i=1}^n \boldsymbol{I}(\boldsymbol{U}_i\leq \boldsymbol{u})
\end{align*}
where $\boldsymbol{U}_i=\frac{1}{n}(R_i^{(1)},\ldots,R_i^{(p)})$ and $R_i^{(k)}=\sum_{j=1}^n I_{X_j^{(k)}\leq X_i^{(k)}}$ (abbreviated as Copula). The authors propose two different versions of multiplier bootstraps to calculate $p-$values. We decided to use the computationally faster one, though the other is known to lead to more powerful tests under small and moderate sample sizes. Note that already the faster version is around 50 times slower than our test whereas the other version is about 3000 times slower. We use the implementation in the npcp-package \citep{npcp}. A test based on the empirical copula posseses power under a broader range of alternatives than the other tests. Covariance, correlation and also huberized covariance test are constructed to have power against changes in the linear dependence, whereas a copula test can detect any kind of change. We therefore expect it to be less powerful than the others if there is indeed a change in the linear dependence.\\
\cite{kojadinovic2016testing} propose to use multivariate generalizations of Spearmans $\rho$. Denote $\hat{\rho}^{i,j}$ Spearmans $\rho$ of $(X^{(i)}_1,\ldots,X^{(i)}_T)$ and $(X^{(j)}_1,\ldots,X^{(j)}_T),$ then the test (abbreviated as Spearman) is based on
\begin{align*}
{2 \choose p}^{-1}\sum_{1\leq i < j\leq p} \hat{\rho}^{i,j}.
\end{align*}
There are two other proposals, but this one seems to have the largest power \citep{kojadinovic2016testing}. There are (at least) two possibilities to estimate the long run variance: a multiplier bootstrap and a kernel estimation. The authors favour the former since the latter has problems to hold it size under strong serial dependence. Nevertheless we choose the kernel estimator since it is considerably faster and we do not look at strong dependences.\\
Two tests based on generalizations of Kendalls $\tau$ are proposed in \cite{quessy2013}. We look at the one which is implemented in the \cite{npcp}-Package (abbreviated as Kendall) and based on
\begin{align*}
\frac{2}{T(T-1)}\sum_{1\leq i<j\leq T}^T \left( I_{\{\boldsymbol{X}_i<\boldsymbol{X}_j\}}+I_{\{\boldsymbol{X}_i>\boldsymbol{X}_j\}}\right).
\end{align*}
Two different estimators for the long run variance are possible, a multiplier bootstrap \citep{bucher2014dependent} and a kernel estimator. We choose the later because of the lower computation times. Note that Spearman and Kendall are type of projection tests. One expects that they have large power if changes are into the same direction (the dependence gets stronger or weaker overall) but low power if changes are into different directions (the dependence between some variables gets stronger but between others gets weaker).\\
We sample from a multivariate AR(1) model
\begin{align}\label{modelmulti}
\boldsymbol{X}_t=\rho \boldsymbol{X}_{t-1}+\boldsymbol{\epsilon}_t,~t=1,\ldots,T
\end{align}
where $(\boldsymbol{\epsilon}_t)_{t\in \mathbb{Z}}$ is a series of independent and multivariate t-distributed random vectors with mean $\boldsymbol{\mu}=0$, shape $V$ and $k$ degrees of freedom. The smaller $k$ the more heavy tailed is $(\boldsymbol{X}_t)_{t\in \mathbb{N}}$. The larger $k$ the more similar is $(\boldsymbol{X}_t)_{t\in \mathbb{N}}$ to a Gaussian distribution. Consequently if we write $k=\infty$ we sample $(\boldsymbol{\epsilon}_t)_{t\in \mathbb{Z}}$ from a multivariate normal distribution.\\
First we want to verify if the tests hold their size. We notice that the multivariate tests get even more conservative under serial dependence if the dimension $p$ increases. The reason behind this is not known yet and content of future research. We try to counterbalance this effect by choosing very short bandwidths $b_T$ which can be found in Table \ref{bandmulti}.
\begin{table}
\begin{center}
\begin{tabular}{l|c}
Estimator&$b_T$\\\hline
Cor&$b_T=log_2(T)$\\
Cov&$b_T=log_{1.8+p(p+1)/40}(T/50)$\\
Huberg&$b_T=log_{1.8+p(p+1)/40}(T/50)$\\
Huberm&$b_T=log_{1.8+p(p+1)/40}(T/50)$\\
Spearman&$b_T=T^{1/3}$\\
Kendall&$b_T=T^{1/3}$\\
Copula&$b_T=log_p(n/12.5)$
\end{tabular}
\caption{\label{bandmulti}Chosen bandwidths $b_T$ for different multivariate change-point tests.}
\end{center}
\end{table}
Note that we choose the somehow arbitrary bandwidths based on plenty of simulations. Since there is no theoretical justification we can only guarantee that these choices are reasonable in our framework of AR(1) processes with moderate $p$ and $\rho.$ Up to our knowledge there are now profound simulation studies published for multivariate cusum-type statistics. \cite{aue2009break} choose $b_T=\log_{10}(T)$, but only consider very weak serial dependence, \cite{wied2014nonparametric} prefer $b_T=\lfloor T^{1/4}\rfloor$ but only looked at MA(1) processes.\\
First we look at the behaviour under the null hypothesis. Results are based on 1000 runs. Under normally distributed innovations we see that all tests hold their size reasonably well, see Table \ref{multnorm}. Multivariate cusum type tests are conservative if $p=5$, $\rho=0.5$ and small $T$. Furthermore the covariance based tests is very anti conservative for small $T$, large $p$ and no serial dependence. There is not much changing if we look at $t_3$ distributed innovations, see Table \ref{multt3}. Only the correlation based test behaves different and gets strongly anti conservative, especially if $p$ is large.\\
We want to look also at the behaviour under the alternative. Note that there are countless different scenarios one can look at. We concentrate on two. First we want to verify if projection type statistics are indeed given an advantage if changes are uniformly into one direction. 
\begin{table}
	\begin{tabular}{l|ccc|ccc||ccc|ccc}
	$\rho$&\multicolumn{6}{c||}{0}&\multicolumn{6}{|c}{0.5}\\
	$p$&\multicolumn{3}{c}{2}&\multicolumn{3}{|c||}{5}&\multicolumn{3}{|c}{2}&\multicolumn{3}{|c}{5}\\
		$T$	&200 &400 &800 &200 &400 &800 &200 &400 &800 &200 &400 &800\\\hline
Cor			&0.04&0.06&0.06&0.00&0.01&0.02&0.07&0.06&0.08&0.02&0.02&0.05\\
Cov			&0.02&0.03&0.04&0.20&0.03&0.02&0.01&0.03&0.04&0.02&0.01&0.02\\
Huberg0		&0.04&0.04&0.06&0.03&0.04&0.04&0.06&0.07&0.06&0.01&0.05&0.06\\
Huberg0.5	&0.04&0.05&0.06&0.03&0.04&0.05&0.05&0.07&0.07&0.01&0.05&0.06\\
Huberg0.8	&0.04&0.06&0.06&0.03&0.04&0.04&0.05&0.07&0.07&0.01&0.05&0.07\\
Huberm0		&0.05&0.05&0.06&0.03&0.04&0.04&0.05&0.05&0.06&0.03&0.04&0.06\\
Huberm0.5	&0.05&0.05&0.05&0.04&0.04&0.03&0.04&0.06&0.06&0.02&0.03&0.06\\
Huberm0.8	&0.05&0.06&0.06&0.03&0.04&0.04&0.06&0.07&0.07&0.01&0.04&0.06\\
Copula		&0.00&0.00&0.23&0.00&0.08&0.30&0.00&0.00&0.22&0.17&0.06&0.23\\
Spearman	&0.03&0.05&0.04&0.03&0.03&0.04&0.04&0.04&0.06&0.03&0.05&0.04\\
Kendall		&0.04&0.05&0.05&0.05&0.04&0.05&0.05&0.05&0.06&0.07&0.06&0.05
\end{tabular}
	\caption{\label{multnorm}Empirical size under multivariate AR(1) model: $\boldsymbol{X}_t=\rho \boldsymbol{X}_{t-1}+\boldsymbol{\epsilon}_t,~t=1,\ldots,T$ with $\rho\in \{0,0.5\}$, $T\in\{200,400,800\}$, $\boldsymbol{\epsilon}\sim N(\boldsymbol{0},I_p)$, $p\in {2,5}$ and a nominal size of $0.05$.}
\end{table}

\begin{table}
	\begin{tabular}{l||ccc|ccc||ccc|ccc}
	$\rho$&\multicolumn{6}{c||}{0}&\multicolumn{6}{|c}{0.5}\\
	$p$&\multicolumn{3}{c}{2}&\multicolumn{3}{|c||}{5}&\multicolumn{3}{|c}{2}&\multicolumn{3}{|c}{5}\\
		$T$	&200 &400 &800 &200 &400 &800 &200 &400 &800 &200 &400 &800\\\hline
Cor			&0.07&0.04&0.04&0.12&0.09&0.07&0.08&0.07&0.07&0.19&0.14&0.13\\
Cov			&0.01&0.01&0.01&0.21&0.04&0.01&0.01&0.01&0.03&0.09&0.02&0.01\\
Huberg0		&0.04&0.06&0.05&0.03&0.04&0.05&0.06&0.06&0.06&0.01&0.05&0.07\\
Huberg0.5	&0.04&0.05&0.04&0.03&0.04&0.06&0.06&0.06&0.07&0.01&0.05&0.08\\
Huberg0.8	&0.04&0.04&0.04&0.03&0.03&0.05&0.06&0.07&0.08&0.01&0.05&0.08\\
Huberm0		&0.05&0.06&0.05&0.04&0.04&0.04&0.06&0.05&0.05&0.03&0.04&0.06\\
Huberm0.5	&0.03&0.05&0.03&0.04&0.04&0.04&0.07&0.05&0.06&0.01&0.03&0.05\\
Huberm0.8	&0.04&0.05&0.04&0.03&0.04&0.04&0.06&0.06&0.07&0.00&0.04&0.08\\
Copula		&0.00&0.00&0.39&0.00&0.07&0.05&0.00&0.00&0.38&0.00&0.08&0.15\\
Spearman	&0.02&0.05&0.04&0.03&0.04&0.05&0.03&0.04&0.06&0.05&0.05&0.06\\
Kendall		&0.02&0.05&0.05&0.05&0.04&0.05&0.05&0.06&0.06&0.07&0.05&0.05

\end{tabular}
	\caption{\label{multt3}Empirical size under multivariate AR(1) model: $\boldsymbol{X}_t=\rho \boldsymbol{X}_{t-1}+\boldsymbol{\epsilon}_t,~t=1,\ldots,T$ with $\rho\in \{0,0.5\}$, $T\in\{200,400,800\}$, $\boldsymbol{\epsilon}\sim t_3(\boldsymbol{0},I_p)$, $p\in {2,5}$ and a theoretical size of $0.05$.}
\end{table}
Therefore we look at model (\ref{modelmulti}) with
\begin{align*}
\epsilon_t\sim \begin{cases}N(0,I_p)&t=1,\ldots,200 \\
N(0,\Sigma_{\Delta})&t=201,\ldots,400
\end{cases}~~\mbox{with}~~\Sigma_{\Delta}=\begin{pmatrix}
1&0.2&\Delta&0.2\\
0.2&1&\Delta&\Delta\\
\Delta&\Delta&1&0.2\\
0.2&\Delta&0.2,1
\end{pmatrix}
\end{align*}
and $\Delta \in [-0.2,0.2].$
So depending on $\Delta$ we have a change into the same direction ($\Delta=0.2$) or into opposite directions ($\Delta=-0.2$). As we can see in Figure \ref{multidiff} projection type tests like Spearman and Kendall outperform truly multivariate ones if the change is into the same direction and have problems to detect a change if it is into different directions. Surprisingly the Copula-test shows similar behaviour. For larger $T$ this test would also be able to detect changes if $\Delta=-0.2$ but the power for $\Delta=0.2$ stays clearly higher. We do not have any explanation for this. All other tests have surprisingly higher power if $\Delta$ is negative. We do not see significant changes of the order of these tests depending on $\Delta.$ The Cor and Huber0.8 are generally the ones with the highest power, slightly outperforming Cov, which is a little more conservative under the null hypothesis.\\
\begin{figure}
\includegraphics[width=0.9\textwidth]{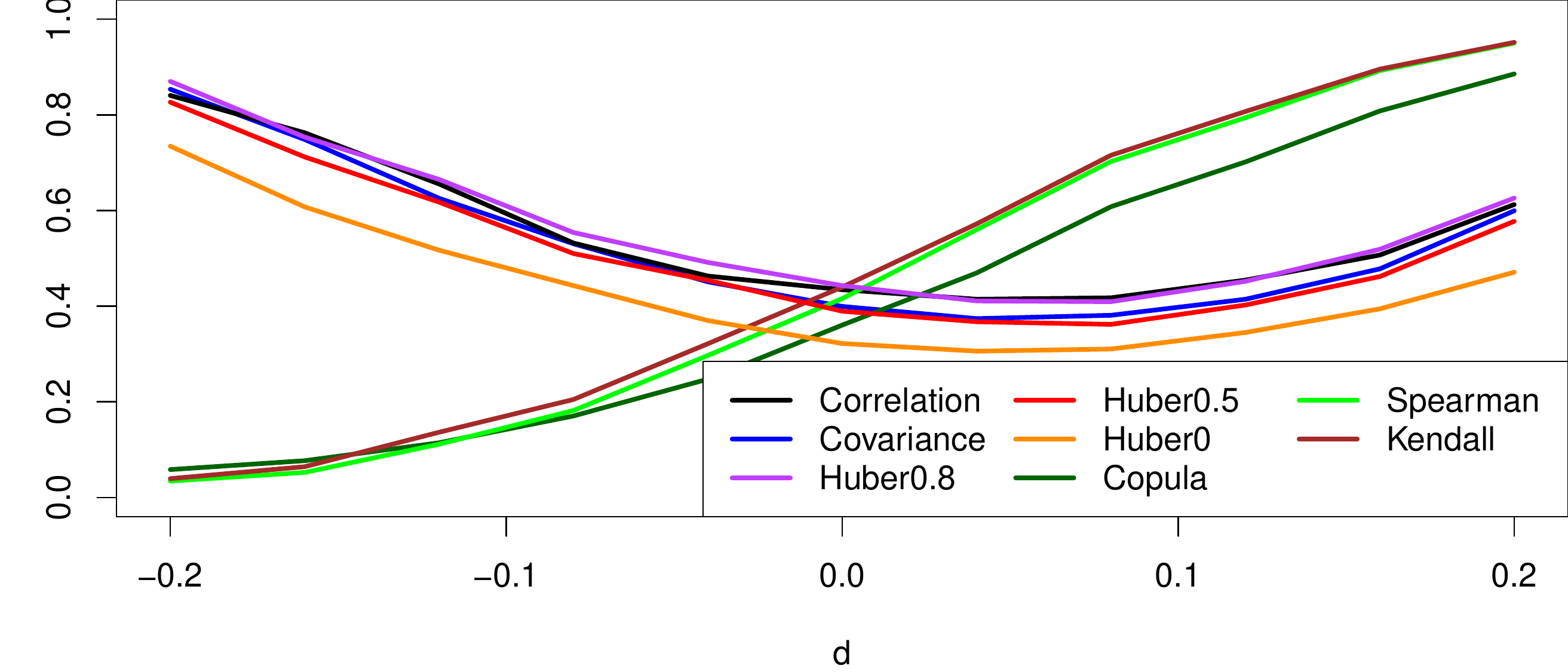}
\caption{Empirical power under one change point and a change of the cross correlation structure of the innovations from $\mbox{Cov}(\boldsymbol{\epsilon}_t)=I_4$ for $t=1,\ldots,200$ to $\mbox{Cov}(\boldsymbol{\epsilon}_t)=\Sigma_{\Delta}$ for $t=201,\ldots,400$.\label{multidiff}}
\end{figure}
We want to investigate the influence of heavy tails. Therefore we look at model (\ref{modelmulti}) with
\begin{align*}
\epsilon_t\sim \begin{cases}t_{df}(\boldsymbol{0},I_p)&t=1,\ldots,200 \\
t_{df}(\boldsymbol{0},\Sigma_2)&t=201,\ldots,400
\end{cases}~~\mbox{with}~~\Sigma_{\Delta}=\begin{pmatrix}
1&0.3&0&0.3\\
0.3&1&0&0\\
0&0&1&0.3\\
0.3&0&0.3,1
\end{pmatrix}
\end{align*}
and $df\in\{1,2,3,5,8,10,20,50,200\}$. Note that we choose the special structure with some covariances changing and others not to achieve a fair comparison between projection type tests and truly multivariate ones. Results are based on 2000 runs. We see in Figure \ref{multidf} that even under 200 degrees of freedom, which is hardly distinguishable from a normal distribution, the Huber08 outperforms all the other tests including Cor. The difference between Huber08 and Cor as well as Cov gets larger for less degrees of freedom. We see that apart from Huber0 all tests loose power if the innovation distribution gets more heavy tailed. But there is a difference how fast the decrease is. The loss for Cov and Cor is the largest and for Huber05 the smallest. The power of the correlation based test increases for 1 degree of freedom since the tests gets extremely anti conservative in this case. Only for very few degrees of freedom Huber0 and Huber0.5 outperform Huber0.8. 
\begin{figure}
\includegraphics[width=0.9\textwidth]{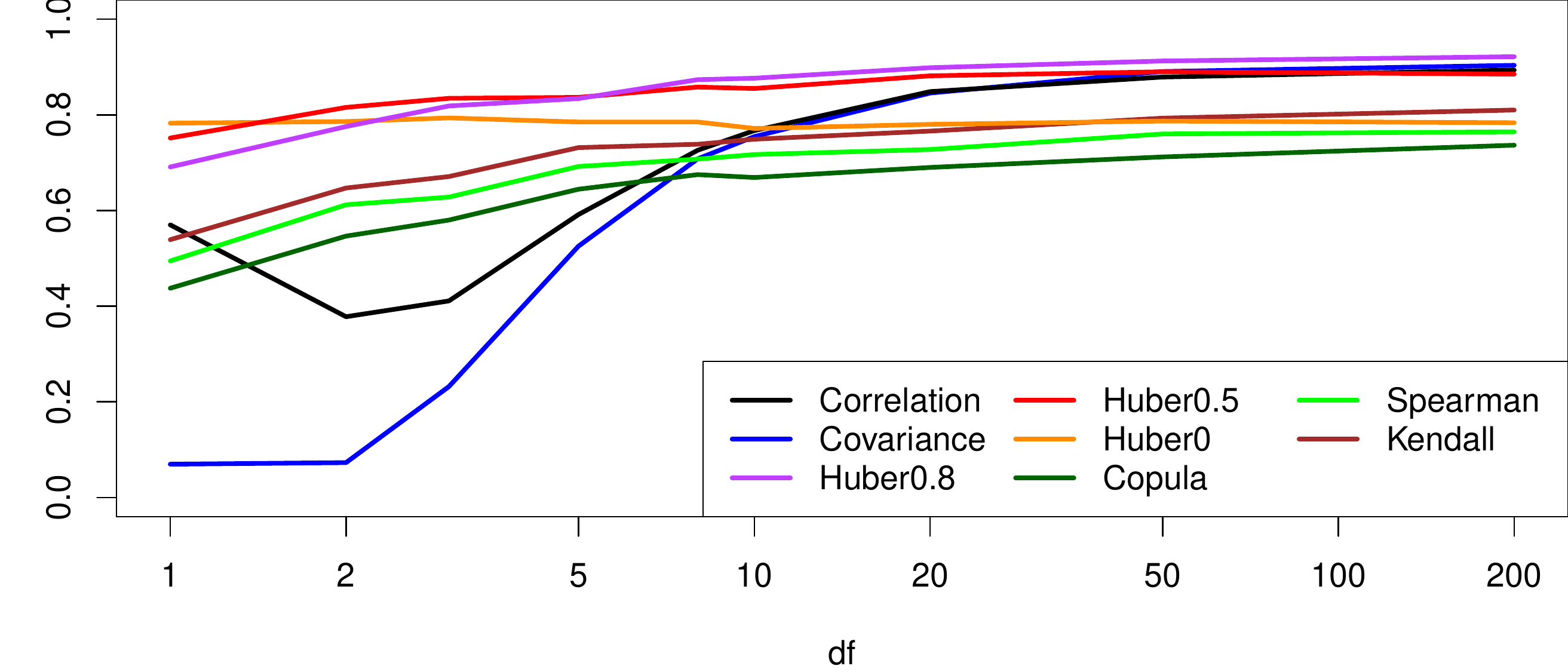}
\caption{Empirical power under one change point and a change of the cross correlation structure of the innovations from $\boldsymbol{\epsilon}_t\sim t_{df}(\boldsymbol{0},I_4)$ for $t=1,\ldots,200$ to $\boldsymbol{\epsilon}_t\sim t_{df}(\boldsymbol{0},\Sigma_2)$  for $t=201,\ldots,400$ and $df\in\{1,2,3,5,8,10,20,50,200\}$.\label{multidf}}
\end{figure}

\section{Summary}
We develop a new non-parametric and robust approach to detect change-points in possibly multivariate time series. The method can be easily adjusted to the type of change one is interested in. We explicitly propose tests for a change in location, scale or cross dependence. 
Simulations indicate that these test are almost as powerful as classical cusum type test under Gaussian data. In these settings they have similar power as other robust methods (if they already exist), while having usually a far lower computational complexity. Under heavy tails our new methods clearly outperform cusum type tests and usually also outperform existing robust methods.\\
Simulations reveal an interesting problem. Our tests become conservative under serial dependence. This effect seems to be more severe if the dimension $p$ of the time series is large. Classical cusum tests show the same behaviour. To the best of our knowledge there are no theoretical results describing this behaviour yet. If it was possible to correct the test statistic, the test would become more powerful under the alternative.\\
It remains to investigate the behaviour of our method under local and fixed alternatives. In the later case the long run covariance estimation does not converge to the theoretical value under stationarity anymore. In the multivariate case the estimated matrix gets even singular which creates an additional challenge.\\
Also of interest are the properties of the related change point estimator. The distribution of the cusum estimator can be found in \cite{csorgo1997limit}. There is also a similar result for a robust change point test estimator of a location change based on ranks \citep{gerstenberger2018robust} which has the same convergence rates as the classical cusum one. The straightforward estimator of change for our method is $\mbox{argmax}~W_T(x)^2$. One would expect that this estimator is more efficient under heavy tailed data but less efficient under Gaussianity.

\bibliographystyle{model2-names}
{\small
\bibliography{lit3}
}

\appendix
\section{Proofs}\label{proofpanel}
We first show that the multivariate cusum statistic converges to a multivariate Brownian motion, if we fix $\boldsymbol{\mu}$ and $\boldsymbol{\sigma}$. Denote by
\begin{align*}
M_n=\frac{1}{T}\left(\sum_{i=1}^{\lfloor nx \rfloor}\boldsymbol{Y}_i-\frac{\lfloor nx \rfloor}{n}\sum_{i=1}^n\boldsymbol{Y}_i\right)
\end{align*}
the multivariate cusum statistic with known location $\boldsymbol{\mu}$ and scale $\boldsymbol{\sigma}.$
\begin{proposition}
	Let $(\boldsymbol{X}_i)_{i\in \mathbb{N}}$ be a stationary and strongly mixing sequence with approximation constants $(a_k)_{k\in \mathbb{N}}$ fulfilling $a_k=O(k^{-1-\epsilon})$ for some $\epsilon>0$ and $\Psi:\mathbb{R}^p \rightarrow \mathbb{R}^s$ be a bounded function such that $U=\mbox{Cov}(\boldsymbol{Y_1},\boldsymbol{Y_1})+2\sum_{h=1}^\infty \mbox{Cov}(\boldsymbol{Y}_i,\boldsymbol{Y}_{i+h})$ has only strictly positive eigenvalues, then
	\begin{align*}
	(M_n(t))_{t\in [0,1]}\stackrel{w}{\rightarrow} (M(t))_{t\in [0,1]}
	\end{align*}
	where $(M_t)_{t\in [0,1]}$ is an $s$-dimensional Brownian motion with $\mbox{Cov}(W_t)=tU.$
\end{proposition}
\begin{proof}
We apply the functional central limit theorem 1.4 of \cite{merlevede2000functional}. For this we first use a Cramer-Wold-device for elements in the s-dimensional Skorochod space, see for example Proposition 4.1 in \cite{wooldridge1988some}. So let $\boldsymbol{\lambda}=\lambda_1,\ldots,\lambda_s$ be arbitrary with $\boldsymbol{\lambda}^T\boldsymbol{\lambda}=1,$ then \begin{align}\label{CW}
(M_n(t))_{t\in [0,1]}\stackrel{w}{\rightarrow} (M(t))_{t\in [0,1]}
 ~~\mbox{if and only if}~~ \left(\sum_{j=1}^s \lambda_j M_n(t)^{(j)}\right)_{t\in [0,1]}\stackrel{w}{\rightarrow} \left(\sum_{j=1}^s \lambda_j M_n(t)^{(j)}\right)_{t\in [0,1]}.\end{align}
 By changing of summation, we see that (\ref{CW}) holds if an invariance principle applies to $\frac{1}{\sqrt{T}}\sum_{i=1}^{Tx}(\boldsymbol{\lambda}^T\boldsymbol{Y}_i),$ which is given in \cite{merlevede2000functional} under three conditions:
 \begin{itemize}
 	\item $\alpha_T T\rightarrow 0,$
 	\item $\sum_{n=1}^\infty \int_0^{a_n} Q^2_{|\boldsymbol{\lambda}^T\boldsymbol{Y}_1|}(u)~du<\infty$ where $Q_W(u)=\inf\{t\geq 0:P(W>t)\leq u\},$
 	\item the long run variance of $(\boldsymbol{\lambda}^T\boldsymbol{Y}_i)_{i \in \mathbb{N}}$ is strictly positive.
 \end{itemize}
 The first condition is satisfied, since $\alpha_k=O(k^{-1-\epsilon})$. For the second we use that $\boldsymbol{\lambda}^T\boldsymbol{Y}_1$ is bounded such that $|Q^2_{|\boldsymbol{\lambda}^T\boldsymbol{Y}_1|}(u)|\leq C$ for some $C>0$. The last condition is fulfilled since $U$ has only positive eigenvalues.
\end{proof}
Now we show that estimating location and scale is asymptotically negligible.
\begin{proposition}
Let Assumptions 1-4 be fulfilled then
\begin{align*}
&(\boldsymbol{R}_{T}(x))_{0\leq x\leq 1}\\
&=\left[\frac{1}{\sqrt{T}}\left(\sum_{t=1}^{\lfloor Tx \rfloor}\boldsymbol{Y}_{i,T}-\frac{\lfloor Tx \rfloor}{T}\sum_{i=1}^T\boldsymbol{Y}_{i,T}\right)-
\frac{1}{\sqrt{T}}\left(\sum_{t=1}^{\lfloor Tx \rfloor}\boldsymbol{Y}_{i}-\frac{\lfloor Tx \rfloor}{T}\sum_{i=1}^T\boldsymbol{Y}_{i}\right)\right]_{0\leq x\leq 1}\rightarrow 0.
\end{align*}
\end{proposition}
\begin{proof}
	The proof consists of two steps. First we show that $R_T(x) \rightarrow 0 ~ \forall x\in [0,1]$ and after that we show that $(R_n(x))_{x \in [0,1]}$ is tight. We use the following covariance inequality
\begin{align}\label{covin}
|\mbox{Cov}(X_1,X_{1+h})-\mathbb{E}(X_1)\mathbb{E}(X_{1+h})|\leq C \alpha_h^{1-1/p-1/q}||X_1||_p||X_{1+h}||_q
\end{align}
for $1/p+1/q=1$ which is derived in \cite{ibragimov1975independent} and \cite{deo1973note} and holds also with a slightly different proof for $p=\infty$, see \cite{borovkova2001limit}.\\
First we note that $\mathbb{E}(R_T(x))=\boldsymbol{0}$ since
\begin{align*}
\frac{1}{\sqrt{T}}\mathbb{E}\left(\sum_{t=1}^{\lfloor Tx \rfloor}\mathbb{E}(\boldsymbol{Y}_{i,T}-\boldsymbol{Y}_{i}|\hat{\boldsymbol{\sigma}},\hat{\boldsymbol{\mu}})-
\frac{\lfloor Tx \rfloor}{T}\sum_{i=1}^T\mathbb{E}(\boldsymbol{Y}_{i,T}-\boldsymbol{Y}_{i}|\hat{\boldsymbol{\sigma}},\hat{\boldsymbol{\mu}})\right)=\boldsymbol{0}.
\end{align*}
Now we decompose $\mathbb{R}^p$ into $A_{\epsilon}(a_1),~A_{\epsilon}(a_n),B_T,{B_1}_T,\ldots,{B_l}_T, {E_1}_T,\ldots, {E_r}$ and the remaining part $G$. We have not defined sets with the subindex $T$ yet. For a set $A$ we define $A_T$ for some $\epsilon>0$ as:
\begin{align*}
{A_T}=\{\mathbf{x}\in \mathbb{R}^p: \exists \mathbf{y}\in A: |\mathbf{x}-\mathbf{y}|\leq T^{-\frac{1}{2}-\epsilon}\}.
\end{align*}
We can find some $\epsilon>0$ and $T_0$ such that these sets are disjoint. Let $i \in \{1,\ldots,s\}$ be arbitrary. Using the $c_r$ inequality several times we can split the expectation of $R^2(x)$ the following way:
\begin{align*}
\mathbb{E}(R^2(x))&\leq \frac{2^q}{T}\sum_{j=1}^n\mathbb{E}([1-\frac{k}{T}]\sum_{t=1}^k [\boldsymbol{Y}_{t,T}^{(i)}-\boldsymbol{Y}_t^{(i)}]I_{\boldsymbol{X}_t\in A_{\epsilon}(a_j)}-\frac{k}{T}\sum_{t=k+1}^T[\boldsymbol{Y}^{(i)}_{t,T}-\boldsymbol{Y}^{(i)}_t]I_{\boldsymbol{X}_t\in A_{\epsilon}(a_j)})^2\\
&+\frac{2^q}{T}\mathbb{E}([1-\frac{k}{T}]\sum_{t=1}^k [\boldsymbol{Y}^{(i)}_{t,T}-\boldsymbol{Y}^{(i)}_t]I_{\boldsymbol{X}_t\in B_T}-\frac{k}{T}\sum_{t=k+1}^T[\boldsymbol{Y}^{(i)}_{t,T}-\boldsymbol{Y}^{(i)}_t]I_{\boldsymbol{X}_t\in B_T})^2\\
&+\frac{2^q}{T}\sum_{j=1}^l\mathbb{E}([1-\frac{k}{T}]\sum_{t=1}^k [\boldsymbol{Y}^{(i)}_{t,T}-\boldsymbol{Y}^{(i)}_t]I_{\boldsymbol{X}_t\in {B_j}_T}-\frac{k}{T}\sum_{t=k+1}^T[\boldsymbol{Y}^{(i)}_{t,T}-\boldsymbol{Y}^{(i)}_t]I_{\boldsymbol{X}_t\in {B_j}_T})^2\\
&+\frac{2^q}{T}\sum_{j=1}^r\mathbb{E}([1-\frac{k}{T}]\sum_{t=1}^k [\boldsymbol{Y}^{(i)}_{t,T}-\boldsymbol{Y}^{(i)}_t]I_{\boldsymbol{X}_t\in {E_j}_T}-\frac{k}{T}\sum_{t=k+1}^T[\boldsymbol{Y}^{(i)}_{t,T}-\boldsymbol{Y}^{(i)}_t]I_{\boldsymbol{X}_t\in {E_j}_T})^2\\
&+\frac{2^q}{T}\mathbb{E}([1-\frac{k}{T}]\sum_{t=1}^k [\boldsymbol{Y}^{(i)}_{t,T}-\boldsymbol{Y}^{(i)}_t]I_{\boldsymbol{X}_t\in G_T}-\frac{k}{T}\sum_{t=k+1}^T[\boldsymbol{Y}^{(i)}_{t,T}-\boldsymbol{Y}^{(i)}_t]I_{\boldsymbol{X}_t\in G_T})^2\\
&=A_1+\ldots A_n+B+B_1+\ldots+B_l+E_1+\ldots +E_r+G
\end{align*}
where $q=n+l+r+2$.
Let without loss of generality $D_{\boldsymbol{\sigma}}=I_p$ and $\boldsymbol{\mu}=\boldsymbol{0}$. First we treat $G$, since $\Psi$ is 2 times continuous differentiable in $G_T$ we use a Taylor expansion of first order: 
\begin{align}\label{Taylor}\Psi\left(D^{-1}_{\hat{\boldsymbol{\sigma}}}[\boldsymbol{X}_i-\hat{\boldsymbol{\mu}}]\right)^{(k)}&=\Psi\left(\boldsymbol{X}_i+ [D_{\hat{\boldsymbol{\sigma}}}-I_p]^{-1}\boldsymbol{X}_i- D_{\hat{\boldsymbol{\sigma}}}\hat{\boldsymbol{\mu}}\right)^{(k)}\\
&=\Psi(\boldsymbol{X}_i)'([D_{\hat{\boldsymbol{\sigma}}}-I_p]^{-1}\boldsymbol{X}_i- D_{\hat{\boldsymbol{\sigma}}}\hat{\boldsymbol{\mu}})\\
&+[(D_{\hat{\boldsymbol{\sigma}}}-I_p)\boldsymbol{X}_i-D_{\hat{\boldsymbol{\sigma}}}\hat{\boldsymbol{\mu}}]^T R_T(\boldsymbol{X}_i)[(D_{\hat{\boldsymbol{\sigma}}}-I_p)\boldsymbol{X}_i-D_{\hat{\boldsymbol{\sigma}}}\hat{\boldsymbol{\mu}}]
\end{align}
where $R_T(\boldsymbol{X}_i)$ depends on $\hat{\boldsymbol{\mu}}$ and $\hat{\boldsymbol{\sigma}}.$ Since $\boldsymbol{x}^T{\Psi^{(i)}}''(\boldsymbol{x})\boldsymbol{x}\leq C_2$, $\boldsymbol{X}_i^TR_T(\boldsymbol{X}_i)\boldsymbol{X}_i$ is bounded. In the following we abbreviate $\frac{\partial \Psi(\boldsymbol{x})^{(i)}}{\partial x_j}|_{\boldsymbol{x}=\boldsymbol{y}}$ as ${\Psi_j^{(i)}}'(\boldsymbol{y})$ 
\begin{align*}
G= \frac{2^q}{T}\mathbb{E}&\left(\left[ \{1-\frac{k}{T}\}\sum_{t=1}^k\Psi'(\boldsymbol{X}_i)^{(i)}I_{\{\boldsymbol{X_t}\in G_T\}}[(D_{\hat{\boldsymbol{\sigma}}}-I_p)\boldsymbol{X}_t-D_{\hat{\boldsymbol{\sigma}}}\hat{\boldsymbol{\mu}}]\right.\right. \\
&+\{1-\frac{k}{T}\}\sum_{t=1}^k [(D_{\hat{\boldsymbol{\sigma}}}-I_p)\boldsymbol{X}_t-D_{\hat{\boldsymbol{\sigma}}}\hat{\boldsymbol{\mu}}]^T R_T(\boldsymbol{X}_t)[(D_{\hat{\boldsymbol{\sigma}}}-I_p)\boldsymbol{X}_t-D_{\hat{\boldsymbol{\sigma}}}\hat{\boldsymbol{\mu}}]I_{\{\boldsymbol{X_t}\in G_T\}}\\
&-\frac{k}{T}\sum_{t=1}^k\Psi'(\boldsymbol{X}_t)^{(i)}I_{\{\boldsymbol{X_t}\in D_T\}}[(D_{\hat{\boldsymbol{\sigma}}}-I_p)\boldsymbol{X}_t-G_{\hat{\boldsymbol{\sigma}}}\hat{\boldsymbol{\mu}}]\\
&\left.\left.-\frac{k}{T}\sum_{t=k+1}^T  [(D_{\hat{\boldsymbol{\sigma}}}-I_p)\boldsymbol{X}_t-D_{\hat{\boldsymbol{\sigma}}}\hat{\boldsymbol{\mu}}]^T R_T(\boldsymbol{X}_t)[(D_{\hat{\boldsymbol{\sigma}}}-I_p)\boldsymbol{X}_t-D_{\hat{\boldsymbol{\sigma}}}\hat{\boldsymbol{\mu}}]I_{\{\boldsymbol{X_t}\in G_T\}}\right]^2\right)\\
&\leq \frac{8\cdot 2^q}{T}\mathbb{E}\left(\left\{\sum_{j=1}^p \left(\frac{1}{\hat{\sigma}}^{(j)}-1\right)\left(\{1-\frac{k}{T}\}\sum_{t=1}^k \left[{\Psi_j^{(i)}}'(\boldsymbol{X}_t)X_t^{(j)}I_{\{\boldsymbol{X}_t\in G_T\}}\right.\right.\right.\right.\\
&-\left. \mathbb{E}\left({\Psi_j^{(i)}}'(\boldsymbol{X}_1)X_1^{(j)}I_{\{\boldsymbol{X}_1\in G_T\}}\right)\right]\\
&\left.\left.\left.-\frac{k}{T}\sum_{t={k+1}}^T \left[{\Psi_j^{(i)}}'(\boldsymbol{X}_t)X_t^{(j)}I_{\{\boldsymbol{X_i}\in G_T\}}-\mathbb{E}\left({\Psi_j^{(i)}}'(\boldsymbol{X}_1)X_1^{(j)}I_{\{\boldsymbol{X_1}\in G_T\}}\right)\right]\right)\right\}^2\right)\\
&+\frac{8 \cdot 2^q}{T}\mathbb{E}\left(\left\{\sum_{j=1}^p \frac{1}{\hat{\sigma}}^{(j)}\hat{\mu}^{(j)}\left(\frac{k}{T}\sum_{t=1}^k \left[{\Psi_j^{(i)}}'(\boldsymbol{X}_t)I_{\{\boldsymbol{X_t}\in G_T\}}-\mathbb{E}\left({\Psi_j^{(i)}}'(\boldsymbol{X}_1)I_{\{\boldsymbol{X_1}\in G_T\}} \right)\right.\right.\right.\right.\\
&\left.\left.\left.-\frac{k}{T}\sum_{i=1}^k \left[{\Psi_j^{(i)}}'(\boldsymbol{X}_t)I_{\{\boldsymbol{X_t}\in G_T\}}-\mathbb{E}\left({\Psi_j^{(i)}}'(\boldsymbol{X}_1)I_{\{\boldsymbol{X_1}\in D_T\}} \right)\right]\right)\right\}^2\right)\\
&+\frac{8\cdot 2^q}{T}\mathbb{E}\left(\left\{\{1-\frac{k}{T}\}\sum_{t=1}^k  [(D_{\hat{\boldsymbol{\sigma}}}-I_p)\boldsymbol{X}_t-D_{\hat{\boldsymbol{\sigma}}}\hat{\boldsymbol{\mu}}]^T R_T(\boldsymbol{X}_t)\right.\right.\\
&\left.\left.[(D_{\hat{\boldsymbol{\sigma}}}-I_p)\boldsymbol{X}_t-D_{\hat{\boldsymbol{\sigma}}}\hat{\boldsymbol{\mu}}]I_{\{\boldsymbol{X_t}\in G_T\}}\right\}^2\right)\\
&+\frac{8\cdot 2^q}{T}\mathbb{E}\left(\left\{\frac{k}{T}\sum_{t=k+1}^T  [(D_{\hat{\boldsymbol{\sigma}}}-I_p)\boldsymbol{X}_t-D_{\hat{\boldsymbol{\sigma}}}\hat{\boldsymbol{\mu}}]^T R_T(\boldsymbol{X}_t)\right.\right.\\
&\left.\left.[(D_{\hat{\boldsymbol{\sigma}}}-I_p)\boldsymbol{X}_t-D_{\hat{\boldsymbol{\sigma}}}\hat{\boldsymbol{\mu}}]I_{\{\boldsymbol{X_t}\in G_T\}}\right\}^2\right)\\
&=G_1+G_2+G_3+G_4.
\end{align*}
For $G_1$ we get
\begin{align*}
G_1&\leq\frac{8\cdot 2^q}{T}2^{p+1}(1- \frac{k}{T})^2\sum_{j=1}^p\mathbb{E}\left(\left\{\left(\frac{1}{\hat{\sigma}}^{(j)}-1\right)\sum_{t=1}^k \left[{\Psi_j^{(i)}}'(\boldsymbol{X}_t)X_t^{(j)}I_{\{\boldsymbol{X_t}\in G_T\}}\right.\right.\right.\\
&\left.\left.\left.- \mathbb{E}\left({\Psi_j^{(i)}}'(\boldsymbol{X}_1)X_1^{(j)}I_{\{\boldsymbol{X}_1\in G_T\}}\right)\right]\right\}^2\right)\\
&+\frac{8\cdot 2^q}{T}2^{p+1}(\frac{k}{T})^2\sum_{j=1}^p\mathbb{E}\left(\left\{\left(\frac{1}{\hat{\sigma}}^{(j)}-1\right)\sum_{t=k+1}^T \left[{\Psi_j^{(i)}}'(\boldsymbol{X}_t)X_t^{(j)}I_{\{\boldsymbol{X_t}\in G_T\}}\right.\right.\right.\\
&\left.\left.\left.- \mathbb{E}\left({\Psi_j^{(i)}}'(\boldsymbol{X}_1)X_1^{(j)}I_{\{\boldsymbol{X_1}\in G_T\}}\right)\right]\right\}^2\right)=G_{1.1}+G_{1.2}.
\end{align*}
In the following we look only at one summand $j\in \{1,\ldots,p\}:$
\begin{align*}
\mathbb{E}&\left(\left\{\left(\frac{1}{\hat{\sigma}}^{(j)}-1\right)\sum_{t=1}^k \left[{\Psi_j^{(i)}}'(\boldsymbol{X}_t)X_t^{(j)}I_{\{\boldsymbol{X}_t\in G_T\}}- \mathbb{E}\left({\Psi_j^{(i)}}'(\boldsymbol{X}_1)X_1^{(j)}I_{\{\boldsymbol{X}_1\in G_T\}}\right)\right]\right\}^2\right)\\
&\leq \sqrt{\frac{1}{T^{1-2\epsilon}}\mathbb{E}\left(\left\{\sum_{t=1}^k \left[{\Psi_j^{(i)}}'(\boldsymbol{X}_t)X_t^{(j)}I_{\{\boldsymbol{X}_i\in G_T\}}- \mathbb{E}\left({\Psi_j^{(i)}}'(\boldsymbol{X}_1)X_1^{(j)}I_{\{\boldsymbol{X}_1\in G_T\}}\right)\right]\right\}^2\right)}\\
&\leq \sqrt{\frac{1}{T^{1-2\epsilon}} C T}.
\end{align*}
Together with the factor $\frac{32}{T}2^{p+1}(1- \frac{k}{T})^2$ this converges to 0. We get the same result for the remaining summands and $G_{1.2}.$ For $G_3$ and $G_4$ we use the boundedness-condition:
\begin{align*}
G_3&=\frac{8\cdot 2^q}{T}\mathbb{E}\left(\left\{\{1-\frac{k}{T}\}\sum_{t=1}^k  [(D_{\hat{\boldsymbol{\sigma}}}-I_p)\boldsymbol{X}_t-D_{\hat{\boldsymbol{\sigma}}}\hat{\boldsymbol{\mu}}]^T R_T(\boldsymbol{X}_t)\right.\right.\\
&\left.\left.[(D_{\hat{\boldsymbol{\sigma}}}-I_p)\boldsymbol{X}_t-D_{\hat{\boldsymbol{\sigma}}}\hat{\boldsymbol{\mu}}]I_{\{\boldsymbol{X_t}\in G_T\}}\right\}^2\right)\\
&\leq \frac{8\cdot 2^q}{T}\left(\sum_{i=1}^k C T^{-1+2\epsilon}\right)\rightarrow 0.
\end{align*}
Now we look at $B:$
\begin{align*}
C&\leq \frac{2^{q+1}}{T}\mathbb{E}\left(\left[\sum_{t=1}^k \{Y_{t,T}^{(i)}-Y_T^{(i)}\} I_{\boldsymbol{X}_t\in B_T}\right]^2\right)+\frac{2^{q+1}}{T}\mathbb{E}\left(\left[\sum_{t=k+1}^T \{Y_{t,T}^{(i)}-Y_T^{(i)}\} I_{\boldsymbol{X}_t\in B_T}\right]^2\right)\\
& \leq \frac{2^{q+1}}{T}\sum_{t=1}^T\sum_{s=1}^T LT^{-\frac{1}{2}+\epsilon}LT^{-\frac{1}{2}+\epsilon}\mathbb{E}(I_{\boldsymbol{X}_s\in B_T}I_{\boldsymbol{X}_t\in B_T})
\leq \frac{2^{q+2}}{T}T^2T^{-\frac{3}{2}+3\epsilon}\rightarrow 0.
\end{align*}
For $j\in \{1,\ldots,n\}$ arbitrary we look at
\begin{align*}
B_j&\leq \frac{2^{q+1}}{T}\mathbb{E}\left(\left[\sum_{t=1}^k \{Y_{t,T}^{(i)}-Y_T^{(i)}\} I_{\boldsymbol{X}_t\in {B_j}_T}\right]^2\right)+\frac{2^{q+1}}{T}\mathbb{E}\left(\left[\sum_{t=k+1}^T \{Y_{t,T}^{(i)}-Y_T^{(i)}\} I_{\boldsymbol{X}_t\in {B_j}_T}\right]^2\right)\\
&=B_{j,1}+B{j,2}
\end{align*}
and get for $B_{j,1}:$
\begin{align*}
B_{j,1}&\leq\frac{2^{q+1}2^p}{T}\sum_{j=1}^p\mathbb{E}\left(\left[\sum_{t=1}^k \left\{\Psi^{(i)}\left(\frac{X_t^{(1)}-\hat{\mu}^{(1)}}{\hat{\sigma}^{(1)}},\ldots,\frac{X_t^{(j)}-\hat{\mu}^{(j)}}{\hat{\sigma}^{(j)}},\right.\right.\right.\right.\\
&\left.\left.\left.\left.\frac{X_t^{(j+1)}-\mu^{(j+1)}}{\sigma^{(j+1)}},\ldots,\frac{X_t^{(p)}-\mu^{(p)}}{\sigma^{(p)}}\right)\right.\right.\right.\\
&\left.\left.\left.-\Psi^{(i)}\left(\frac{X_t^{(1)}-\hat{\mu}^{(1)}}{\hat{\sigma}^{(1)}},\ldots,\frac{X_t^{(j-1)}-\hat{\mu}^{(j-1)}}{\hat{\sigma}^{(j-1)}},\frac{X_t^{(j)}-\mu^{(j)}}{\sigma^{(j)}},\ldots,\frac{X_t^{(p)}-\mu^{(p)}}{\sigma^{(p)}}\right)\right\}I_{\boldsymbol{X}_t\in {B_j}_T}\right]^2\right)\\
&\leq \frac{2^{q+p+1}}{T}pk^2LT^{-\frac{3}{2}+3\epsilon}\rightarrow 0.
\end{align*}
For $j\in \{1,\ldots,n\}$ arbitrary we look at $A_j$
\begin{align*}
A_j&\leq  \frac{2^{q+1}}{T}\mathbb{E}\left(\left[\sum_{t=1}^k \{\boldsymbol{Y}^{(i)}_{t,T}-\boldsymbol{Y}^{(i)}_T\} I_{\boldsymbol{X}_t\in A_\epsilon(a_j)}\right]^2\right)\\
&+\frac{2^{q+1}}{T}\mathbb{E}\left(\left[\sum_{t=k+1}^T \{\boldsymbol{Y}^{(i)}_{t,T}-\boldsymbol{Y}^{(i)}_T\} I_{\boldsymbol{X}_t\in A_\epsilon(a_j)}\right]^2\right)=A_{j,1}+A_{j,2}
\end{align*}
where for $A_{j,1}$ we have
\begin{align*}
\leq & \frac{2^{q+1}}{T}\mathbb{E}\left(\left[\sum_{t=1}^k \{\boldsymbol{Y}^{(i)}_{t,T}-\boldsymbol{Y}^{(i)}_T\} I_{|\boldsymbol{X}_t-a_j|\leq T^{-\frac{1}{2p}-\delta_1}}\right]^2\right)\\
=&\frac{2^{q+1}}{T}\mathbb{E}\left(\left[\sum_{t=1}^k \{\boldsymbol{Y}^{(i)}_{t,T}-\boldsymbol{Y}^{(i)}_T\} I_{T^{-\frac{1}{2p}-\delta_1}<|\boldsymbol{X}_t-a_j|<\epsilon}\right]^2\right)=A_{j,1,1}+A_{j,1,2}
\end{align*}
For $A_{j,1,1}$ we use that $P(|\boldsymbol{X}_t-a_j|\leq T^{-\frac{1}{2p}-\delta_1})\leq C \cdot T^{-\frac{1}{2}-\delta_1}$ to show that $A_{j,1,1}\rightarrow 0:$
\begin{align*}
\frac{2^{q+1}}{T}&\mathbb{E}\left(\left[\sum_{t=1}^k \{\boldsymbol{Y}^{(i)}_{t,T}-\boldsymbol{Y}^{(i)}_T\} I_{|\boldsymbol{X}_t-A|\leq T^{-\frac{1}{2p}-\delta_1}}\right]^2\right)
\\
&\leq \frac{2^{q+1}}{T}\mathbb{E}\left(\left[\sum_{t=1}^k 2K I_{|\boldsymbol{X}_t-a_j|\leq T^{-\frac{1}{2p}-\delta_1}}\right]^2\right)\\
&\leq \frac{2^{q+2}K^2}{T}\mathbb{E}\left(\left[\sum_{t=1}^k I_{|\boldsymbol{X}_t-a_j|\leq T^{-\frac{1}{2p}-\delta_1}}-P(|\boldsymbol{X}_1-a_j|\leq T^{-\frac{1}{2p}-\delta_1})\right]^2\right)\\
&+\frac{2^{q+2}K^2k^2}{T}P(|\boldsymbol{X}_1-a_j|\leq T^{-\frac{1}{2p}-\delta_1})^2\\
&\leq \frac{2^{q+2}K^2}{T}C_1 C_2\cdot T^{-1-\delta_1}+\frac{2^{q+2}K^2k^2}{T}C_2\cdot T^{-1-\delta_1}\rightarrow 0
\end{align*}
For $A_{1,2}$ we use a Taylor-decomposition and get analogously to the decomposition of $D,$ $A_{1,2,1},~A_{1,2,2},~A_{1,2,3}$ and $A_{1,2,4}$ where we split $A_{1,2,1}$ into $A_{1,2,1,1}$ and $A_{1,2,1,2}$ similar to $D_1$ which is split into $D_{1,1}$ and $D_{1,2}.$ So  
$A_{1,2,1,1}$ has the following form:
\begin{align*}
A_{1,2,1,1}&=\frac{2^{q+4}}{T}2^{p+1}(1- \frac{k}{T})^2\sum_{m=1}^p\mathbb{E}\left(\left\{\left(\frac{1}{\hat{\sigma}}^{(m)}-1\right)\sum_{t=1}^k \left[
{\Psi^{(i)}_m}'(\boldsymbol{X}_t)X_t^{(m)}I_{\{T^{-\frac{1}{2p}-\delta_1}<|\boldsymbol{X}_t-a_j|<\epsilon\}}\right.\right.\right.\\
&\left.\left.\left.- \mathbb{E}\left({\Psi^{(i)}_m}'(\boldsymbol{X}_1)X_1^{(j)}I_{\{T^{-\frac{1}{2p}-\delta_1}<|\boldsymbol{X}_1-a_j|<\epsilon\}}\right)\right]\right\}^2\right)
\end{align*}
For every summand $m$ we get
\begin{align*}
\mathbb{E}&\left(\left\{\left(\frac{1}{\hat{\sigma}}^{(m)}-1\right)\sum_{t=1}^k \left[{\Psi^{(i)}_m}'(\boldsymbol{X}_t)X_t^{(m)}I_{\{T^{-\frac{1}{2p}-\delta_1}<|\boldsymbol{X}_t-a_j|<\epsilon\}}\right.\right.\right.\\
&\left.\left.\left.- \mathbb{E}\left({\Psi^{(i)}_m}'(\boldsymbol{X}_1)X_1^{(j)}I_{\{T^{-\frac{1}{2p}-\delta_1}<|\boldsymbol{X}_1-a_j|<\epsilon\}}\right)\right]\right\}^2\right)\\
&\leq \left[\frac{1}{T^{1-2\delta_2}}\mathbb{E}\left(\left\{\sum_{t=1}^k \left[{\Psi^{(i)}_m}'(\boldsymbol{X}_t)X_t^{(m)}I_{\{T^{-\frac{1}{2p}-\delta_1}<|\boldsymbol{X}_t-a_j|<\epsilon\}}\right.\right.\right.\right.\\
&\left.\left.\left.\left.- \mathbb{E}\left({\Psi^{(i)}_m}'(\boldsymbol{X}_1)X_1^{(m)}I_{\{T^{-\frac{1}{2p}-\delta_1}<|\boldsymbol{X}_1-a_j|<\epsilon\}}\right)\right]\right\}^2\right)\right]^\frac{1}{2}\\
&\leq \sqrt{\frac{1}{T^{1-2\delta_2}} C T \mathbb{E}\left(\left[{\Psi^{(i)}_m}'(\boldsymbol{X}_t)X_1^{(m)}I_{\{T^{-\frac{1}{2p}-\delta_1}<|\boldsymbol{X}_1|<\epsilon\}}\right]^2\right)}\\
&\leq \sqrt{\frac{1}{T^{1-2\epsilon}} C T^{2-\delta}}.
\end{align*}
Together with the remaining factor  $\frac{2^{q+4}}{T}2^{p+1}(1-\frac{k}{T})^2$  this converges to 0. Analogous calculations shows convergence for the other summands.\\
Finally we look at $E_1,\ldots,E_r$. To shorten notation we assume $E_1=\{\boldsymbol{x}\in \mathbb{R}^p:x_1=0\}.$ Define by $N_{s}=\{x \in \mathbb{R}: \frac{s-1}{T^{\frac{3}{2}}}\leq x <\frac{s}{T^{\frac{3}{2}}}\}$ for $s=-T,\ldots,T$, then we decompose $E_1$ further:
\begin{align*}
\frac{2^q}{T}&\mathbb{E}\left([1-\frac{k}{T}]\sum_{t=1}^k [\boldsymbol{Y}_{t,T}^{(i)}-\boldsymbol{Y}_t^{(i)}]I_{Y_t\in E_{1,T}}-\frac{k}{T}\sum_{t=k+1}^T[\boldsymbol{Y}_{t,T}^{(i)}-\boldsymbol{Y}_t^{(i)}]I_{Y_t\in E_{1,T}}\right)^2\\
&=\frac{2^q}{T}\mathbb{E}\left([1-\frac{k}{T}]\sum_{s=-T}^T I_{\hat{\mu}^{(1)}\in N_s}\left(\sum_{t=1}^k \sum_{v=-T}^T [\boldsymbol{Y}_{t,T}^{(i)}-\boldsymbol{Y}^{(i)}_t]I_{\boldsymbol{X}_t\in N_v}\right.\right.\\
&\left.\left.-\frac{k}{T}\sum_{t=k+1}^T\sum_{s=-T}^T[\boldsymbol{Y}^{(i)}_{t,T}-\boldsymbol{Y}_t^{(i)}]I_{\boldsymbol{X}_t\in N_v}\right)\right)^2\\
&=\frac{2^q}{T}\mathbb{E}([1-\frac{k}{T}]\sum_{s=-T}^TI_{\hat{\mu}\in N_s}\left(\sum_{t=1}^k \sum_{v=-T}^T [\boldsymbol{Y}^{(i)}_{t,T}-\boldsymbol{Y}^{(i)}_t]I_{\boldsymbol{X}_t\in N_v}\right.\\
&-\mathbb{E}\left([\boldsymbol{Y}_{1,T}^{(i)}-\boldsymbol{Y}^{(i)}_1]I_{\boldsymbol{X}_1\in N_v}\right)\\
&\left.\left.-\frac{k}{T}\sum_{t=k+1}^T\sum_{v=-T}^T[\boldsymbol{Y}^{(i)}_{t,T}-\boldsymbol{Y}^{(i)}_t]I_{\boldsymbol{X}_t\in N_v}-\mathbb{E}\left([\boldsymbol{Y}^{(i)}_{1,T}-\boldsymbol{Y}^{(i)}_1]I_{\boldsymbol{X}_1\in N_v}\right)\right)\right)^2\\
&\leq  \frac{2^q}{T}\mathbb{E}\left(\sum_{s=-T}^T I_{\hat{\mu}^{(1)}\in N_v}\left(\sum_{t=1}^k \sum_{v=-T}^T [\boldsymbol{Y}^{(i)}_{t,T}-\boldsymbol{Y}^{(i)}_t]I_{\boldsymbol{X}_t\in N_v}-\mathbb{E}\left([\boldsymbol{Y}_{1,T}^{(i)}-\boldsymbol{Y}^{(i)}_1]I_{\boldsymbol{X}_1\in N_v}\right)\right)\right)^2\\
+&\frac{2^q}{T}\mathbb{E}\left(\sum_{s=-T}^TI_{\hat{\mu}^{(1)}\in N_s}\left(\sum_{t=k+1}^T \sum_{v=-T}^T [\boldsymbol{Y}^{(i)}_{t,T}-\boldsymbol{Y}^{(i)}_t]I_{\boldsymbol{X}_t\in N_v}-\mathbb{E}\left([\boldsymbol{Y}^{(i)}_{1,T}-\boldsymbol{Y}^{(i)}_1]I_{\boldsymbol{X}_1\in N_v}\right)\right)\right)^2\\
&=C_1+C_2
\end{align*}
Depending in which sector $\hat{\boldsymbol{\mu}}$ and $\boldsymbol{X}_t$ lie most summands are 0. So for example $C_1$ simplifies to
\begin{align*}
&\frac{2^{q+1}}{T}\mathbb{E}\left(\sum_{s=-T}^TI_{\hat{\mu}^{(1)} \in N_s}\left(\sum_{t=1}^k \sum_{v=-T}^T [\boldsymbol{Y}_{t,T}-\boldsymbol{Y}_t]I_{\boldsymbol{X}_t\in N_v}-\mathbb{E}\left([\boldsymbol{Y}_{1,T}-\boldsymbol{Y}_1]I_{\boldsymbol{X}_1\in N_v}\right)\right)\right)^2\\
&\leq \frac{2^{q+1}4C^2}{T}\mathbb{E}\left(\sum_{s=-T}^TI_{\hat{\mu}^{(1)} \in N_s}\left(\sum_{t=1}^kI_{\boldsymbol{X}_t\in N_{s-1}\cup N_s\cup N_{s+1}}\right)\right)^2\\
&= \frac{2^{q+1}4C^2}{T}\mathbb{E}\left(\sum_{s}^T\sum_{t_1,t_2=1}^kI_{\hat{\mu}^{(1)} \in N_{s}}I_{\boldsymbol{X}_{t_1}\in N_{s-1}\cup N_{s}\cup N_{s+1}}I_{\boldsymbol{X}_{t_2}\in N_{s-1}\cup N_{s}\cup N_{s+1}}\right)\\
&= \frac{2^{q+1}36C^2}{T}\mathbb{E}\left(\sum_{s}^T\sum_{t_1,t_2=1}^kI_{\boldsymbol{X}_{t_1}\in E_{1,T}}I_{\boldsymbol{X}_{t_2}\in E_{1,T}}\right)\\
&\leq \frac{2^{q+1}36C^2} T^{-\frac{1}{2}+\epsilon}T^{-\frac{1}{2}+\epsilon}\rightarrow 0
\end{align*}
Now we want to show that $(\boldsymbol{R}_T(x))_{x\in [0,1]}$ respectively every component of it, is tight. We use the moment criteria by \cite{Billingsley} with $\gamma=4$ and $\alpha=2$ so we have to show that
\begin{align*}
\mathbb{E}(|R_n(y)^{(i)}-R_n(x)^{(i)}|^4)\leq C_5|y-x|^2,~\forall~x,y\in [0,1],~T\in \mathbb{N}.
\end{align*}
We denote $k=\lfloor Tx\rfloor>j=\lfloor Ty \rfloor$ and follow the same ideas as before, namely splitting the sums by the $c_r$ inequality into $A_1+\ldots A_n+B+B_1+\ldots+B_l+E_1+\ldots +E_r+G$ and treating them as before. The only difference now is that we have to look at 4-th moments instead of 2nd moments.
\begin{align*}
\mathbb{E}&|R_T(y)^{(i)}-R_T(x)^{(i)}|^4\\
&=\frac{1}{T^2}\mathbb{E}\bigg|\left(1-\frac{k}{T}\right)\sum_{t=1}^k\left( Y_{t,T}^{(i)}-Y_{t}^{(i)}\right)-\frac{k}{T}\sum_{t=k+1}^T\left( Y_{t,T}^{(i)}-Y_t^{(i)}\right)\\
&-\left(1-\frac{j}{T}\right)\sum_{t=1}^j \left(Y_{t,T}^{(i)}-Y_{t}^{(i)}\right)+\frac{j}{T}\sum_{t=j+1}^T \left(Y_{t,T}^{(i)}-Y_t^{(i)}\right)\bigg|^4\\
&\leq\frac{8^q}{T^2}\sum_{j=1}^n\mathbb{E}\bigg|\left(1-\frac{k}{T}\right)\sum_{t=1}^k\left( Y_{t,T}^{(i)}-Y_{t}^{(i)}\right)I_{\{X_t \in A_T\}}-\frac{k}{T}\sum_{t=k+1}^T\left( Y_{t,T}^{(i)}-Y_t^{(i)}\right)I_{\{X_t \in A_{\epsilon}(a_j)\}}\\
&-\left(1-\frac{j}{T}\right)\sum_{t=1}^j \left(Y_{t,T}^{(i)}-Y_{t}^{(i)}\right)I_{\{X_t \in A_T\}}+\frac{j}{T}\sum_{t=j+1}^T \left(Y_{t,T}^{(i)}-Y_t^{(i)}\right)I_{\{X_t \in A_{\epsilon}(a_j)\}}\bigg|^4\\
&\leq\frac{8^q}{T^2}\mathbb{E}\bigg|\left(1-\frac{k}{T}\right)\sum_{t=1}^k\left( Y_{t,T}^{(i)}-Y_{t}^{(i)}\right)I_{\{X_t \in B_T\}}-\frac{k}{T}\sum_{t=k+1}^T\left( Y_{t,T}^{(i)}-Y_t^{(i)}\right)I_{\{X_t \in B_T\}}\\
&-\left(1-\frac{j}{T}\right)\sum_{t=1}^j \left(Y_{t,T}^{(i)}-Y_{t}^{(i)}\right)I_{\{X_t \in B_T\}}+\frac{j}{T}\sum_{t=j+1}^T \left(Y_{t,T}^{(i)}-Y_t^{(i)}\right)I_{\{X_t \in B_T\}}\bigg|^4\\
&\leq\frac{8^q}{T^2}\sum_{j=1}^l\mathbb{E}\bigg|\left(1-\frac{k}{T}\right)\sum_{t=1}^k\left( Y_{t,T}^{(i)}-Y_{t}^{(i)}\right)I_{\{X_t \in B_{j_T}\}}-\frac{k}{T}\sum_{t=k+1}^T\left( Y_{t,T}^{(i)}-Y_t^{(i)}\right)I_{\{X_t \in B_{j_T}\}}\\
&-\left(1-\frac{j}{T}\right)\sum_{t=1}^j \left(Y_{t,T}^{(i)}-Y_{t}^{(i)}\right)I_{\{X_t \in E_{j_T}\}}+\frac{j}{T}\sum_{t=j+1}^T \left(Y_{t,T}^{(i)}-Y_t^{(i)}\right)I_{\{X_t \in E_{j_T}\}}\bigg|^4\\
&\leq\frac{8^q}{T^2}\sum_{j=1}^l\mathbb{E}\bigg|\left(1-\frac{k}{T}\right)\sum_{t=1}^k\left( Y_{t,T}^{(i)}-Y_{t}^{(i)}\right)I_{\{X_t \in E_{j_T}\}}-\frac{k}{T}\sum_{t=k+1}^T\left( Y_{t,T}^{(i)}-Y_t^{(i)}\right)I_{\{X_t \in E_{j_T}\}}\\
&-\left(1-\frac{j}{T}\right)\sum_{t=1}^j \left(Y_{t,T}^{(i)}-Y_{t}^{(i)}\right)I_{\{X_t \in E_{j_T}\}}+\frac{j}{T}\sum_{t=j+1}^T \left(Y_{t,T}^{(i)}-Y_t^{(i)}\right)I_{\{X_t \in E_{j_T}\}}\bigg|^4\\
&\leq\frac{8^q}{T^2}\mathbb{E}\bigg|\left(1-\frac{k}{T}\right)\sum_{t=1}^k\left( Y_{t,T}^{(i)}-Y_{t}^{(i)}\right)I_{\{X_t \in G_T\}}-\frac{k}{T}\sum_{t=k+1}^T\left( Y_{t,T}^{(i)}-Y_t^{(i)}\right)I_{\{X_t \in G_T\}}\\
&-\left(1-\frac{j}{T}\right)\sum_{t=1}^j \left(Y_{t,T}^{(i)}-Y_{t}^{(i)}\right)I_{\{X_t \in G_T\}}+\frac{j}{T}\sum_{t=j+1}^T \left(Y_{t,T}^{(i)}-Y_t^{(i)}\right)I_{\{X_t \in G_T\}}\bigg|^4\\
&=\frac{8^q}{T^2}(A_1+\ldots A_n+B+B_1+\ldots+B_l+E_1+\ldots +E_r+G).
\end{align*} 
We look exemplarily at $G$:
\begin{align*}
G&\leq 8^{12}\left[\frac{j-k}{T}\right]^4\mathbb{E}\left(\left\{\sum_{j=1}^p \left(\frac{1}{\hat{\sigma}^{(j)}}-1\right)\sum_{t=1}^k \left[{\Psi^{(i)}_j}'(\boldsymbol{X}_t)X_t^{(j)}I_{\{\boldsymbol{X_i}\in G_T\}}\right.\right.\right.\\
&-\left.\left.\left. \mathbb{E}\left({\Psi^{(i)}_j}'(\boldsymbol{X}_1)X_1^{(j)}I_{\{\boldsymbol{X_1}\in D_T\}}\right)\right]\right\}^4\right)\\
&+8^{12}\left[\frac{T-j}{T}\right]^4\mathbb{E}\left(\left\{\sum_{j=1}^p \left(\frac{1}{\hat{\sigma}^{(j)}}-1\right)\sum_{t=k+1}^j \left[{\Psi^{(i)}_j}'(\boldsymbol{X}_t)X_t^{(j)}I_{\{\boldsymbol{X_i}\in G_T\}}\right.\right.\right.\\
&-\left.\left.\left. \mathbb{E}\left({\Psi^{(i)}_j}'(\boldsymbol{X}_1)X_1^{(j)}I_{\{\boldsymbol{X_1}\in G_T\}}\right)\right]\right\}^4\right)\\
&+8^{12}\left[\frac{j-k}{T}\right]^4\mathbb{E}\left(\left\{\sum_{j=1}^p \left(\frac{1}{\hat{\sigma}^{(j)}}-1\right)\sum_{t=j+1}^T \left[{\Psi^{(i)}_j}'(\boldsymbol{X}_t)I_{\{\boldsymbol{X_t}\in G_T\}}\right.\right.\right.\\
&-\left.\left.\left. \mathbb{E}\left({\Psi^{(i)}_j}'(\boldsymbol{X}_1)I_{\{\boldsymbol{X_1}\in G_T\}}\right)\right]\right\}^4\right)\\
&+8^{12}\left[\frac{k}{T}\right]^4\mathbb{E}\left(\left\{\sum_{j=1}^p \left(\frac{1}{\hat{\sigma}^{(j)}}-1\right)\sum_{t=k+1}^j \left[{\Psi^{(i)}_j}'(\boldsymbol{X}_t)X_t^{(j)}I_{\{\boldsymbol{X}_i\in G_T\}}\right.\right.\right.\\
&-\left.\left.\left. \mathbb{E}\left({\Psi^{(i)}_j}'(\boldsymbol{X}_1)X_1^{(j)}I_{\{\boldsymbol{X}_1\in G_T\}}\right)\right]\right\}^4\right)\\
&+8^{12}\left[\frac{j-k}{T}\right]^4\mathbb{E}\left(\left\{\sum_{j=1}^p \frac{1}{\hat{\sigma}}^{(j)}\hat{\mu}^{(j)}\right.\right.\\
&\left.\left.\left(\frac{k}{T}\sum_{t=1}^k \left[{\Psi^{(i)}_j}'(\boldsymbol{X}_t)I_{\{\boldsymbol{X_i}\in G_T\}}-\mathbb{E}\left({\Psi^{(i)}_j}'(\boldsymbol{X}_1)I_{\{\boldsymbol{X_1}\in G_T\}} \right)\right]\right)\right\}^4\right)\\
&+8^{12}\left[\frac{T-j}{T}\right]^4\mathbb{E}\left(\left\{\sum_{j=1}^p \frac{1}{\hat{\sigma}}^{(j)}\hat{\mu}^{(j)}\right.\right.\\
&\left.\left.\left(\frac{k}{T}\sum_{t=k+1}^j \left[{\Psi^{(i)}_j}'(\boldsymbol{X}_t)I_{\{\boldsymbol{X}_t\in G_T\}}-\mathbb{E}\left({\Psi^{(i)}_j}'(\boldsymbol{X}_t)I_{\{\boldsymbol{X}_1\in G_T\}} \right)\right]\right)\right\}^4\right)\\
&+8^{12}\left[\frac{j-k}{T}\right]^4\mathbb{E}\left(\left\{\sum_{j=1}^p \frac{1}{\hat{\sigma}}^{(j)}\hat{\mu}^{(j)}\right.\right.\\
&\left.\left.\left(\frac{k}{T}\sum_{t=j+1}^T \left[{\Psi^{(i)}_j}'(\boldsymbol{X}_t)I_{\{\boldsymbol{X}_t\in G_T\}}-\mathbb{E}\left({\Psi^{(i)}_j}'(\boldsymbol{X}_t)I_{\{\boldsymbol{X}_1\in G_T\}} \right)\right]\right)\right\}^4\right)\\
&+8^{12}\left[\frac{k}{T}\right]^4\mathbb{E}\left(\left\{\sum_{j=1}^p \frac{1}{\hat{\sigma}}^{(j)}\hat{\mu}^{(j)}\left(\frac{k}{T}\sum_{t=k+1}^j \left[{\Psi^{(i)}_j}'(\boldsymbol{X}_t)I_{\{\boldsymbol{X}_t\in G_T\}}-\mathbb{E}\left({\Psi^{(i)}_j}'(\boldsymbol{X}_1)I_{\{\boldsymbol{X}_1\in G_T\}} \right)\right]\right)\right\}^4\right)\\
&+8^{12}\left[\frac{j-k}{T}\right]^4\mathbb{E}\left(\left\{\sum_{t=1}^k  [(D_{\hat{\boldsymbol{\sigma}}}-I_p)\boldsymbol{X}_t-D_{\hat{\boldsymbol{\sigma}}}\hat{\boldsymbol{\mu}}]^T R_T(\boldsymbol{X}_t)[(D_{\hat{\boldsymbol{\sigma}}}-I_p)\boldsymbol{X}_t-D_{\hat{\boldsymbol{\sigma}}}\hat{\boldsymbol{\mu}}]I_{\{\boldsymbol{X}_t\in G_T\}}\right\}^4\right)\displaybreak\\
&+8^{12}\left[\frac{T-j}{T}\right]^4\mathbb{E}\left(\left\{\sum_{t=k+1}^j  [(D_{\hat{\boldsymbol{\sigma}}}-I_p)\boldsymbol{X}_t-D_{\hat{\boldsymbol{\sigma}}}\hat{\boldsymbol{\mu}}]^T R_T(\boldsymbol{X}_t)[(D_{\hat{\boldsymbol{\sigma}}}-I_p)\boldsymbol{X}_t-D_{\hat{\boldsymbol{\sigma}}}\hat{\boldsymbol{\mu}}]I_{\{\boldsymbol{X}_t\in G_T\}}\right\}^4\right)\\
&+8^{12}\left[\frac{j-k}{T}\right]^4\mathbb{E}\left(\left\{\sum_{t=j+1}^T  [(D_{\hat{\boldsymbol{\sigma}}}-I_p)\boldsymbol{X}_t-D_{\hat{\boldsymbol{\sigma}}}\hat{\boldsymbol{\mu}}]^T R_T(\boldsymbol{X}_t)[(D_{\hat{\boldsymbol{\sigma}}}-I_p)\boldsymbol{X}_t-D_{\hat{\boldsymbol{\sigma}}}\hat{\boldsymbol{\mu}}]I_{\{\boldsymbol{X}_t\in G_T\}}\right\}^4\right)\\
&+8^{12}\left[\frac{k}{T}\right]^4\mathbb{E}\left(\left\{\sum_{t=k+1}^j  [(D_{\hat{\boldsymbol{\sigma}}}-I_p)\boldsymbol{X}_t-D_{\hat{\boldsymbol{\sigma}}}\hat{\boldsymbol{\mu}}]^T R_T(\boldsymbol{X}_t)[(D_{\hat{\boldsymbol{\sigma}}}-I_p)\boldsymbol{X}_t-D_{\hat{\boldsymbol{\sigma}}}\hat{\boldsymbol{\mu}}]I_{\{\boldsymbol{X}_t\in G_T\}}\right\}^4\right)\\
&=8^12\sum_{i=1}^{12} TA_i
\end{align*}
For $TA_1,\ldots,TA_{12}$ one applies the 4th-moment inequality for strongly mixing sequences \citep{merlevede2000functional}. In case of $TA_1$ this yields:
\begin{align*}
TA1\leq (x-y)^2
2^{3p}\sqrt{(T^{-\frac{1}{2}+\epsilon})^4+k^2D}
\end{align*}
and together with the prefactor $1/T^2$ we get the desired bound containing $(x-y)^2.$ Ror the remainder terms of the Taylor series we use the boundedness and get for example for $TA_{12}:$
\begin{align*}
TA_{12}\leq (j-k)^2T^2D(T^{1-2\epsilon})^4
\end{align*} 
which also results together with the prefactor $1/T^2$ an upper bound containing $(x-y)^2$.
\end{proof}
Now we show consistency of the long run variance estimator $\hat{U}:$
\begin{proposition}
Under Assumptions 1-5, $\hat{U}$ is weakly consistent for $U.$
\end{proposition}
\begin{proof}
Consistency of the long run variance estimator under known $\boldsymbol{\mu}$ and $\boldsymbol{\sigma}$ is known under very general conditions, see \cite{andrews1991heteroskedasticity} and \cite{jong2000consistency}. If $r$-th moments exists, it is required that mixing coefficients $(a_k)_{k\in \mathbb{N}}$ fulfill $a_k=O(k^{r/(r-2)})$  (Theorem 2.1 of \cite{jong2000consistency}) which is covered by Assumption 1. In \cite{jong2000consistency} even the case of estimated standardization is covered as long as $\Psi$ is differentiable and some technical conditions are fulfilled. Since we also want work with discontinuous $\Psi$ we proof the asymptotic negligibility of the estimation of $\boldsymbol{\mu}$ and $\boldsymbol{\sigma}.$ Therefore we use the following decomposition
\begin{align*}
\frac{1}{T}\sum_{s,t}& \left(Y_{s,T}^{(i)}-\frac{1}{T}\sum_{u=1}^T Y_{u,T}^{(i)}\right)
\left(Y_{t,T}^{(j)}-\frac{1}{T}\sum_{u=1}^T Y_{u,T}^{(j)}\right)k\left(\frac{s-t}{b_T}\right)\\
&-\frac{1}{T}\sum_{s,t}\left(Y_s^{(i)}-\mathbb{E}[Y_1]^{(i)}\right)\left(Y_s^{(j)}-\mathbb{E}[Y_1]^{(j)}\right)k\left(\frac{s-t}{b_T}\right)\\
&=\frac{1}{T}\sum_{s,t} \left(Y_{s,T}^{(i)}-\frac{1}{T}\sum_{u=1}^T Y_{u,T}^{(i)}\right)\left(Y_{t,T}^{(j)}-Y_{t}^{(j)}-\frac{1}{T}\sum_{u=1}^T Y_{u,T}^{(j)}-Y_u^{(j)}\right)k\left(\frac{s-t}{b_T}\right)\\
&+\frac{1}{T}\sum_{s,t} \left(Y_{t}^{(j)}-\frac{1}{T}\sum_{u=1}^T Y_{T}^{(j)}\right)\left(Y_{s,T}^{(i)}-Y_{s}^{(i)}-\frac{1}{T}\sum_{u=1}^T Y_{u,T}^{(i)}-Y_u^{(i)}\right)k\left(\frac{s-t}{b_T}\right)\\
&+\frac{1}{T}\sum_{s,t} \left(Y_{s}^{(i)}-\frac{1}{T}\sum_{u=1}^T Y_{T}^{(i)}\right)\left(\mathbb{E}[Y^{(j)}_1]-\frac{1}{T}\sum_{u=1}^T Y_{T}^{(j)}\right)k\left(\frac{s-t}{b_T}\right)\\
&+\frac{1}{T}\sum_{s,t} \left(Y_{t}^{(j)}-\mathbb{E}[Y_1^{(j)}]\right)\left(\mathbb{E}[Y^{(i)}_1]-\frac{1}{T}\sum_{u=1}^T Y_{T}^{(i)}\right)k\left(\frac{s-t}{b_T}\right)\\
&=K_1+K_2+K_3+K_4.
\end{align*}
Every summand has to converge against 0 in probability. Like in the proof of Theorem 1 we expand them by indicator functions. Since $P(A_T^C)\rightarrow 0$ we assume in the following the case of $A_T$. An expansion of $K_1$ yields
\begin{align*}
K_1&=\frac{1}{T}\sum_{s,t} \left(Y_{s,T}^{(i)}-\frac{1}{T}\sum_{u=1}^T Y_{u,T}^{(i)}\right)I_{X_s,X_u\in A_{\epsilon}(a_1)}\left(Y_{t,T}^{(j)}-Y_{t}^{(j)}-\frac{1}{T}\sum_{v=1}^T Y_{v,T}^{(j)}-Y_u^{(j)}\right)\\
&I_{X_t,X_v\in A_{\epsilon}(a_1)}k\left(\frac{s-t}{b_T}\right)\\
&+\frac{1}{T}\sum_{s,t} \left(Y_{s,T}^{(i)}-\frac{1}{T}\sum_{u=1}^T Y_{u,T}^{(i)}\right)I_{X_s,X_u\in A_{\epsilon}(a_1)}\left(Y_{t,T}^{(j)}-Y_{t}^{(j)}-\frac{1}{T}\sum_{v=1}^T Y_{v,T}^{(j)}-Y_u^{(j)}\right)\\
&I_{X_t,X_v\in A_{\epsilon}(a_2)}k\left(\frac{s-t}{b_T}\right)\\
&\vdots\\
&+\frac{1}{T}\sum_{s,t} \left(Y_{s,T}^{(i)}-\frac{1}{T}\sum_{u=1}^T Y_{u,T}^{(i)}\right)I_{X_s,X_u\in G_{T}}\left(Y_{t,T}^{(j)}-Y_{t}^{(j)}-\frac{1}{T}\sum_{v=1}^T Y_{v,T}^{(j)}-Y_u^{(j)}\right)\\
&I_{X_t,X_v\in G_{T}}k\left(\frac{s-t}{b_T}\right)\\
&=K_{1A_1A_1}+K_{1A_1A_2}+K_{1A_1A_3}+\ldots+K_{1GG}
\end{align*}
Let us look exemplarily at $K_{1GG}$ where we use the Taylor expansion (\ref{Taylor}). We omit the indicator functions in the following to improve readability:
\begin{align*}
K_{1GG}&=\frac{1}{T}\sum_{m=1}^p\left(1-\frac{1}{\hat{\sigma}^{(m)}}\right)\sum_{s,t=1}^T \left[Y_s^{(i)}-\frac{1}{T}\sum_{u=1}^T Y_u^{(i)}\right]\\
&\left[{\Psi_m^{(j)}}'(\boldsymbol{X}_t)X_t^{(m)}-\frac{1}{T}\sum_{u=1}^T{\Psi_m^{(j)}}'(\boldsymbol{X}_u)X_u^{(m)} \right]k\left(\frac{s-t}{b_T}\right)\\
&+\frac{1}{T}\sum_{m=1}^p\frac{\hat{\mu}^{(j)}}{\hat{\sigma}^{(j)}}\sum_{s,t=1}^T \left[Y_s^{(i)}-\frac{1}{T}\sum_{u=1}^T Y_u^{(i)}\right]\\
&\left[{\Psi_m^{(j)}}'(\boldsymbol{X}_t)-\frac{1}{T}\sum_{u=1}^T
{\Psi_m^{(j)}}'(\boldsymbol{X}_u) \right]k\left(\frac{s-t}{b_T}\right)\\
+&\frac{1}{T}\sum_{s,t=1}^T \left[Y_s^{(i)}-\frac{1}{T}\sum_{u=1}^T Y_u^{(i)}\right]\left\{[(D_{\hat{\boldsymbol{\sigma}}}-I_p)\boldsymbol{X}_t-D_{\hat{\boldsymbol{\sigma}}}\hat{\boldsymbol{\mu}}]^T R_T(\boldsymbol{X}_t)[(D_{\hat{\boldsymbol{\sigma}}}-I_p)\boldsymbol{X}_t-D_{\hat{\boldsymbol{\sigma}}}\hat{\boldsymbol{\mu}}] \right.\\
&\left.-\frac{1}{T}\sum_{u=1}^T [(D_{\hat{\boldsymbol{\sigma}}}-I_p)\boldsymbol{X}_u-D_{\hat{\boldsymbol{\sigma}}}\hat{\boldsymbol{\mu}}]^T R_T(\boldsymbol{X}_u)[(D_{\hat{\boldsymbol{\sigma}}}-I_p)\boldsymbol{X}_u-D_{\hat{\boldsymbol{\sigma}}}\hat{\boldsymbol{\mu}}]\right\}k\left(\frac{s-t}{b_T}\right)\\
&+\frac{1}{T}\sum_{n,m=1}^p\left(1-\frac{1}{\hat{\sigma}^{(m)}}\right)\left(1-\frac{1}{\hat{\sigma}^{(n)}}\right)
\sum_{s,t=1}^T \left[{\Psi_m^{(i)}}'(\boldsymbol{X}_s)X_s^{(i)}-\frac{1}{T}\sum_{u=1}^T{\Psi_m^{(i)}}'(\boldsymbol{X}_u)X_u^{(i)} \right]\displaybreak\\
&\left[{\Psi_n^{(j)}}'(\boldsymbol{X}_t)X_t^{(n)}-\frac{1}{T}\sum_{u=1}^T{\Psi_n^{(j)}}'(\boldsymbol{X}_u)X_n^{(j)} \right]k\left(\frac{s-t}{b_T}\right)\\
&+\frac{1}{T}\sum_{n,m=1}^p\left(1-\frac{1}{\hat{\sigma}^{(m)}}\right)\frac{\hat{\mu}^{(n)}}{\hat{\sigma}^{(n)}}
\sum_{s,t=1}^T \left[{\Psi_m^{(i)}}'(\boldsymbol{X}_s)X_s^{(i)}-\frac{1}{T}\sum_{u=1}^T{\Psi_m^{(i)}}'(\boldsymbol{X}_u)X_u^{(i)} \right]\\
&\left[{\Psi_n^{(j)}}'(\boldsymbol{X}_t)-\frac{1}{T}\sum_{u=1}^T{\Psi_n^{(j)}}'(\boldsymbol{X}_u) \right]k\left(\frac{s-t}{b_T}\right)\\
&+\frac{1}{T}\sum_{m=1}^p\left(1-\frac{1}{\hat{\sigma}^{(m)}}\right)\left[{\Psi_m^{(i)}}'(\boldsymbol{X}_s)X_s^{(i)}-\frac{1}{T}\sum_{u=1}^T{\Psi_m^{(i)}}'(\boldsymbol{X}_u)X_u^{(i)} \right]\\
&\left\{[(D_{\hat{\boldsymbol{\sigma}}}-I_p)\boldsymbol{X}_t-D_{\hat{\boldsymbol{\sigma}}}\hat{\boldsymbol{\mu}}]^T R_T(\boldsymbol{X}_t)[(D_{\hat{\boldsymbol{\sigma}}}-I_p)\boldsymbol{X}_t-D_{\hat{\boldsymbol{\sigma}}}\hat{\boldsymbol{\mu}}] \right.\\
&\left.-\frac{1}{T}\sum_{u=1}^T [(D_{\hat{\boldsymbol{\sigma}}}-I_p)\boldsymbol{X}_u-D_{\hat{\boldsymbol{\sigma}}}\hat{\boldsymbol{\mu}}]^T R_T(\boldsymbol{X}_u)[(D_{\hat{\boldsymbol{\sigma}}}-I_p)\boldsymbol{X}_u-D_{\hat{\boldsymbol{\sigma}}}\hat{\boldsymbol{\mu}}]\right\}k\left(\frac{s-t}{b_T}\right)\\
&+\frac{1}{T}\sum_{n,m=1}^p\frac{\hat{\mu}^{(m)}}{\hat{\sigma}^{(m)}}\left(1-\frac{1}{\hat{\sigma}^{(n)}}\right)
\sum_{s,t=1}^T \left[{\Psi_m^{(i)}}'(\boldsymbol{X}_s)-\frac{1}{T}\sum_{u=1}^T{\Psi_m^{(i)}}'(\boldsymbol{X}_u) \right]\\
&\left[{\Psi_n^{(j)}}'(\boldsymbol{X}_t)-\frac{1}{T}\sum_{u=1}^T{\Psi_n^{(j)}}'(\boldsymbol{X}_t)X_u^{(n)} \right]k\left(\frac{s-t}{b_T}\right)\\
&+\frac{1}{T}\sum_{n,m=1}^p\frac{\hat{\mu}^{(m)}}{\hat{\sigma}^{(m)}}\frac{\hat{\mu}^{(n)}}{\hat{\sigma}^{(n)}}
\sum_{s,t=1}^T \left[{\Psi_m^{(i)}}'(\boldsymbol{X}_s)-\frac{1}{T}\sum_{u=1}^T{\Psi_m^{(i)}}'(\boldsymbol{X}_u) \right]\\
&\left[{\Psi_n^{(j)}}'(\boldsymbol{X}_t)-\frac{1}{T}\sum_{u=1}^T{\Psi_n^{(j)}}'(\boldsymbol{X}_u) \right]k\left(\frac{s-t}{b_T}\right)\\
&+\frac{1}{T}\sum_{m=1}^p\frac{\hat{\mu}^{(m)}}{\hat{\sigma}^{(m)}}
\sum_{s,t=1}^T \left[{\Psi_m^{(i)}}'(\boldsymbol{X}_s)-\frac{1}{T}\sum_{u=1}^T{\Psi_m^{(i)}}'(\boldsymbol{X}_u) \right]\\
&\left\{[(D_{\hat{\boldsymbol{\sigma}}}-I_p)\boldsymbol{X}_t-D_{\hat{\boldsymbol{\sigma}}}\hat{\boldsymbol{\mu}}]^T R_T(\boldsymbol{X}_t)[(D_{\hat{\boldsymbol{\sigma}}}-I_p)\boldsymbol{X}_t-D_{\hat{\boldsymbol{\sigma}}}\hat{\boldsymbol{\mu}}] \right.\\
&\left.-\frac{1}{T}\sum_{u=1}^T [(D_{\hat{\boldsymbol{\sigma}}}-I_p)\boldsymbol{X}_u-D_{\hat{\boldsymbol{\sigma}}}\hat{\boldsymbol{\mu}}]^T R_T(\boldsymbol{X}_u)[(D_{\hat{\boldsymbol{\sigma}}}-I_p)\boldsymbol{X}_u-D_{\hat{\boldsymbol{\sigma}}}\hat{\boldsymbol{\mu}}]\right\}k\left(\frac{s-t}{b_T}\right)\\
&+\frac{1}{T}\sum_{m=1}^p\left(1-\frac{1}{\hat{\sigma}^{(m)}}\right)\sum_{s,t=1}^T\left\{[(D_{\hat{\boldsymbol{\sigma}}}-I_p)\boldsymbol{X}_t-D_{\hat{\boldsymbol{\sigma}}}\hat{\boldsymbol{\mu}}]^T R_T(\boldsymbol{X}_s)[(D_{\hat{\boldsymbol{\sigma}}}-I_p)\boldsymbol{X}_t-D_{\hat{\boldsymbol{\sigma}}}\hat{\boldsymbol{\mu}}] \right.\\
&\left.-\frac{1}{T}\sum_{u=1}^T [(D_{\hat{\boldsymbol{\sigma}}}-I_p)\boldsymbol{X}_u-D_{\hat{\boldsymbol{\sigma}}}\hat{\boldsymbol{\mu}}]^T R_T(\boldsymbol{X}_u)[(D_{\hat{\boldsymbol{\sigma}}}-I_p)\boldsymbol{X}_u-D_{\hat{\boldsymbol{\sigma}}}\hat{\boldsymbol{\mu}}]\right\}k\left(\frac{s-t}{b_T}\right)\\
&\left[{\Psi_m^{(j)}}'(\boldsymbol{X}_t)X_t^{(m)}-\frac{1}{T}\sum_{u=1}^T{\Psi_m^{(j)}}'(\boldsymbol{X}_u)X_u^{(m)} \right]k\left(\frac{s-t}{b_T}\right)\\\displaybreak
&+\frac{1}{T}\sum_{m=1}^p\frac{\hat{\mu}^{(m)}}{\hat{\sigma}^{(m)}}\sum_{s,t=1}^T\left\{[(D_{\hat{\boldsymbol{\sigma}}}-I_p)\boldsymbol{X}_t-D_{\hat{\boldsymbol{\sigma}}}\hat{\boldsymbol{\mu}}]^T R_T(\boldsymbol{X}_s)[(D_{\hat{\boldsymbol{\sigma}}}-I_p)\boldsymbol{X}_t-D_{\hat{\boldsymbol{\sigma}}}\hat{\boldsymbol{\mu}}] \right.\\
&\left.-\frac{1}{T}\sum_{u=1}^T [(D_{\hat{\boldsymbol{\sigma}}}-I_p)\boldsymbol{X}_u-D_{\hat{\boldsymbol{\sigma}}}\hat{\boldsymbol{\mu}}]^T R_T(\boldsymbol{X}_u)[(D_{\hat{\boldsymbol{\sigma}}}-I_p)\boldsymbol{X}_u-D_{\hat{\boldsymbol{\sigma}}}\hat{\boldsymbol{\mu}}]\right\}k\left(\frac{s-t}{b_T}\right)\\
&\left[{\Psi_m^{(j)}}'(\boldsymbol{X}_t)-\frac{1}{T}\sum_{u=1}^T{\Psi_m^{(j)}}'(\boldsymbol{X}_u) \right]k\left(\frac{s-t}{b_T}\right)\\
&+\frac{1}{T}\sum_{s,t=1}^T\left\{[(D_{\hat{\boldsymbol{\sigma}}}-I_p)\boldsymbol{X}_t-D_{\hat{\boldsymbol{\sigma}}}\hat{\boldsymbol{\mu}}]^T R_T(\boldsymbol{X}_s)[(D_{\hat{\boldsymbol{\sigma}}}-I_p)\boldsymbol{X}_t-D_{\hat{\boldsymbol{\sigma}}}\hat{\boldsymbol{\mu}}] \right.\\
&\left.-\frac{1}{T}\sum_{u=1}^T [(D_{\hat{\boldsymbol{\sigma}}}-I_p)\boldsymbol{X}_u-D_{\hat{\boldsymbol{\sigma}}}\hat{\boldsymbol{\mu}}]^T R_T(\boldsymbol{X}_u)[(D_{\hat{\boldsymbol{\sigma}}}-I_p)\boldsymbol{X}_u-D_{\hat{\boldsymbol{\sigma}}}\hat{\boldsymbol{\mu}}]\right\}\\
&\left\{[(D_{\hat{\boldsymbol{\sigma}}}-I_p)\boldsymbol{X}_t-D_{\hat{\boldsymbol{\sigma}}}\hat{\boldsymbol{\mu}}]^T R_T(\boldsymbol{X}_t)[(D_{\hat{\boldsymbol{\sigma}}}-I_p)\boldsymbol{X}_t-D_{\hat{\boldsymbol{\sigma}}}\hat{\boldsymbol{\mu}}] \right.\\
&\left.-\frac{1}{T}\sum_{u=1}^T [(D_{\hat{\boldsymbol{\sigma}}}-I_p)\boldsymbol{X}_u-D_{\hat{\boldsymbol{\sigma}}}\hat{\boldsymbol{\mu}}]^T R_T(\boldsymbol{X}_u)[(D_{\hat{\boldsymbol{\sigma}}}-I_p)\boldsymbol{X}_u-D_{\hat{\boldsymbol{\sigma}}}\hat{\boldsymbol{\mu}}]\right\}k\left(\frac{s-t}{b_T}\right)\\
\end{align*}
The first three summands are the ones converging most slowly to 0. We examplarily look at the first summand for an arbitrary $m\in \{1,\ldots,p\}.$
\begin{align*}
\frac{1}{T}&\left(1-\frac{1}{\hat{\sigma}^{(m)}}\right)\sum_{s,t=1}^T \left[Y_s^{(i)}-\frac{1}{T}\sum_{u=1}^T Y_u^{(i)}\right]\\
&\left[\Psi^{(j)}_m(\boldsymbol{X}_t)X_t^{(m)}-\frac{1}{T}\sum_{u=1}^T\Psi^{(j)}_m(\boldsymbol{X}_u)X_u^{(m)} \right]k\left(\frac{s-t}{b_T}\right)\\
&=\frac{1}{T}\left(1-\frac{1}{\hat{\sigma}^{(m)}}\right)\sum_{s,t=1}^T \left[Y_s^{(i)}-\mathbb{E}(Y_1^{(i)})-\frac{1}{T}\sum_{u=1}^T Y_u^{(i)}-\mathbb{E}(Y_1^{(i)})\right]\\
&\left[\Psi^{(j)}_m(\boldsymbol{X}_t)X_t^{(m)}-\mathbb{E}\left\{\Psi^{(j)}_m(\boldsymbol{X}_1)X_1^{(m)}\right\}\right.\\
&\left.-\frac{1}{T}\sum_{u=1}^T\Psi^{(j)}_m(\boldsymbol{X}_u)X_u^{(m)}-\mathbb{E}\left\{\Psi^{(j)}_m(\boldsymbol{X}_1)X_1^{(m)}\right\} \right]k\left(\frac{s-t}{b_T}\right)\\
\end{align*} 
To shorten Notation, we assume that $\mathbb{E}(Y_1^{(i)})=0$, $\mathbb{E}\{\psi_m^{(i)}(\boldsymbol{X}_t)X_t^{(m)}\}=0$. A further expansion yields
\begin{align*}
\frac{1}{T}&\left(1-\frac{1}{\hat{\sigma}^{(m)}}\right)\sum_{s,t=1}^T \left[Y_s^{(i)}-\frac{1}{T}\sum_{u=1}^T Y_u^{(i)}\right]\left[\psi_m^{(j)}(\boldsymbol{X}_t)X_t^{(m)}-\frac{1}{T}\sum_{u=1}^T\psi_m^{(j)}(\boldsymbol{X}_u)X_u^{(m)} \right]k\left(\frac{s-t}{b_T}\right)\\
&=\left(1-\frac{1}{\hat{\sigma}^{(m)}}\right)\frac{1}{T}\sum_{s,t}Y_s^{(i)}\psi_m^{(j)}(\boldsymbol{X}_t)X_t^{(m)}k\left(\frac{s-t}{b_T}\right)\\
&-\left(1-\frac{1}{\hat{\sigma}^{(m)}}\right)\frac{1}{T}\sum_{s,t}Y_s^{(i)}k\left(\frac{s-t}{b_T}\right)\frac{1}{T}\sum_{u=1}^T\psi_m^{(j)}(\boldsymbol{X}_u)X_u^{(m)}\\
&-\left(1-\frac{1}{\hat{\sigma}^{(m)}}\right)\frac{1}{T}\sum_{s,t}\psi_m^{(j)}(\boldsymbol{X}_s)X_s^{(m)}k\left(\frac{s-t}{b_T}\right)\frac{1}{T}\sum_{u=1}^TX_u^{(i)}\\
&+\left(1-\frac{1}{\hat{\sigma}^{(m)}}\right)\frac{1}{T}\sum_{s,t}k\left(\frac{s-t}{b_T}\right)\frac{1}{T}\sum_{u=1}^T\psi_m^{(j)}(\boldsymbol{X}_u)X_u^{(m)}\frac{1}{T}\sum_{v=1}^T X_v^{(j)}=I+II+III+IV\\
\end{align*}
The summands $II$,$III$ and $IV$ contain at least one arithmetic mean converging against 0. Therefore we only have to look at $I:$  
\begin{align*}
\mathbb{E}(|I|)&=\mathbb{E}\left\{\left(1-\frac{1}{\hat{\sigma}^{(m)}}\right)\sum_{s,t=1}^T \left[Y_s^{(i)}-\frac{1}{T}\sum_{u=1}^T Y_u^{(i)}\right]\right.\\
&\left.\left[\psi_m^{(j)}(\boldsymbol{X}_t)X_t^{(m)}-\frac{1}{T}\sum_{u=1}^T\psi_m^{(j)}(\boldsymbol{X}_u)X_u^{(m)} \right]k\left(\frac{s-t}{b_T}\right)\right\}\\
&\leq \frac{1}{T^{\frac{1}{2}-\epsilon}}\frac{1}{T}\left\{\sum_{s,t=1}^T\mathbb{E}|X_s^{(i)}\psi_m^{(j)}(\boldsymbol{X}_t)X_t^{(m)}|\right\}\\
&\leq \frac{1}{T^{\frac{1}{2}-\epsilon}}2\sum_{k=0}^{\infty} a_k\rightarrow 0
\end{align*}
which completes the proof.
\end{proof}

\end{document}